\documentclass{elsarticle}
\usepackage{amsmath, amsfonts}
\usepackage[amsthm]{ntheorem}
\usepackage{amssymb}
\usepackage{enumerate}

\newtheorem{thm}{Theorem}
\newtheorem{lem}[thm]{Lemma}
\newtheorem{cor}[thm]{Corollary}
\newtheorem{prop}[thm]{Proposition}
\theoremstyle{definition}
\newtheorem{definition}{Definition}
\newtheorem{remark}{Remark}
\theoremstyle{remark}
\newtheorem{example}{Example}

\usepackage{chngcntr}
\newcounter{prethma}
\counterwithin{thma}{prethma}

\newcommand{\bs}[1]{\boldsymbol{\mathrm{#1}}}  
\newcommand{\kfrac}[2]{\frac{#1}{#2}\genfrac{}{}{0pt}{}{}{+}} 
\newcommand{\lowfrac}[1]{\genfrac{}{}{0pt}{}{}{#1}}  
\renewcommand{\bmod}{\mathrm{mod} \ }

\DeclareMathOperator{\sgn}{sgn}

\usepackage{tikz}
\usetikzlibrary{arrows,positioning,scopes,trees}

\title{Combinatorics of Continuants of Continued Fractions with $3$ Limits}
\date{October 27, 2021}

\author[1]{Douglas Bowman%
\corref{cor1}}
\author[1]{Herman D. Schaumburg}
\address[1]{Northern Illinois University\\1425 West Lincoln Hwy\\DeKalb, IL, 60115\\
dbowman@niu.edu}
\cortext[cor1]{Corresponding author}

\begin{document}

\begin{keyword}
Divergent continued fractions, Continuants, Euler-Minding Theorem, Linear recurrences, Integer partitions \MSC[2010] 05A19 \sep 11B39 \sep 40A15
\end{keyword}

\begin{abstract}
We give combinatorial descriptions of the terms occurring in continuants of general continued fractions that diverge to three limits. Equating this combinatorics with the usual combinatorial 
description due to Euler induces nontrivial identities. Special cases and applications to counting sequences are given.   
\end{abstract}

\maketitle

\section{Overview}

Research on divergent continued fractions usually occurs  in the study of
analytic continued fractions. Meanwhile,
combinatorial aspects of continued fractions are typically studied in
in the  field of enumerative combinatorics.
In this paper we bring the two subjects together and give a combinatorial description
of the continuants of a general class of continued fractions that
diverge to three limits. This class was previously studied from the analytic point of view by
the first author \cite{Bowman2007}.
We are able to relate our combinatorially described polynomials to the classical
continuant polynomials
 going back to Euler. This yields
identities that have a flavor similar to the
identities
 between different bases of symmetric polynomials in as much as there is considerable
cancellation occurring between the monomials on one side, but not the other. 

As usual we write a continued fraction: 

\begin{equation*}
b_0+\cfrac{a_1}{b_1+\cfrac{a_2}{b_2+\cfrac{a_3}{b_3\vphantom{\cfrac{a_4}{b_4}}+\lowfrac{{\raisebox{-1.5ex}{\rotatebox{-4}{$\ddots$}}}}}}}
\end{equation*}
with the more compact notation
\begin{equation}
b_0+\kfrac{a_1}{b_1}\kfrac{a_2}{b_2}\kfrac{a_3}{b_3}\lowfrac{\cdots}.
\label{eqn:CfracNotation}
\end{equation}
The $k$th \emph{classical numerator} $A_k$, and $k$th \emph{classical denominator} $B_k$, of the continued fraction \eqref{eqn:CfracNotation} are the respective numerator and denominator when the finite continued fraction
\[\dfrac{A_k}{B_k}=b_0+\kfrac{a_1}{b_1}\kfrac{a_2}{b_2}\kfrac{a_3}{b_3}\lowfrac{\cdots}\lowfrac{+}\dfrac{a_k}{b_k}\]
is simplified in the usual way. The polynomials $A_k=A_k(a_1,\dots,a_k;b_0,\dots,b_k)$ are also
known as \emph{continuants}. Since $B_k=A_{k-1}(a_2,\dots,a_k;b_1,\dots,b_k)$, it suffices to consider
just the sequence $A_k$.

\subsection{Continuants}

A combinatorial description for the terms of polynomials $A_k$ was first given in 1764 by Euler \cite{Euler1764} in the case where $a_i=1$, for $1\leq i\leq k$. The case where $b_i=1$, for $0\leq i\leq k$ was considered by Sylvester \cite{Sylvester1854} in 1854. The general case was finally given by
Minding \cite{Minding1869} in 1869. See also Chrystal \cite{Chrystal1964Book} and Muir \cite{Muir1960Book}.

This description is simplest in the special case when the indeterminates $b_i$ 
are set equal to unity. There is really no loss of generality due to the simple identity
\begin{equation*}
b_0+\kfrac{a_1}{b_1}\kfrac{a_2}{b_2}\lowfrac{\cdots}\lowfrac{+}\dfrac{a_k}{b_k}
=b_0\left(1+\kfrac{a_1/b_0b_1}{1}\kfrac{a_2/b_1b_2}{1}\lowfrac{\cdots}\lowfrac{+}\dfrac{a_k/b_{k-1}b_k}{1}
\right).
\end{equation*}
 Euler's combinatorial description \cite{Euler1764} is sometimes referred to by the terms \emph{Euler brackets} or \emph{Euler's rule}; see, for example, Davenport \cite{Davenport1992Book} or
Roberts \cite{Roberts1977Book}.
In any event the resulting theorem is known as the Euler-Minding Theorem. 

\begin{thm}[Euler-Minding Theorem, Sylvester's form]\label{EMS}
The classical numerators and denominators of 
\begin{equation}
1+\kfrac{a_1}{1}\kfrac{a_2}{1}\kfrac{a_3}{1}\lowfrac{\cdots} \label{eqn:NumEulerMindingFrac}
\end{equation}
are given by
\begin{equation}
A_k=1 + \sum_{\substack{ k\geq h_1 >^2 h_2 >^2 \cdots >^2 h_\ell \geq 1 \\ \ell \geq 1}} a_{h_1} a_{h_2} \cdots a_{h_\ell},
\label{eqn:NumEulerMindingNum}
\end{equation}
and
\begin{equation}
B_k= 1 + \sum_{\substack{k\geq h_1 >^2 h_2 >^2 \cdots >^2 h_\ell \geq 2 \\ \ell \geq 1}} a_{h_1} a_{h_2} \cdots a_{h_\ell},
\label{eqn:NumEulerMindingDen}
\end{equation}
where $i>^2j$ means $i$ and $j$ have minimal difference $2$; $i\geq j+2$.
\label{thm:NumEulerMinding}
\end{thm}

Thus the monomials in $A_k$ and $B_k$ are described by sequences $h_i$ of the form 
\[
k\geq h_1 >^2 h_2 >^2 \cdots >^2 h_\ell.
\]
We call a sequence satisfying this inequality chain 
a \emph{minimal difference $2$ sequence}.

Note that when $k\to\infty$ limits for $A_k$ and $B_k$ exist in the ring
of formal power series over the monoid
generated by the indeterminates $a_i$. As we will soon see, this does not necessarily hold
for other continued fractions with indeterminate
elements.

\subsection{Divergent Continued Fractions with Multiple Limits} 

Apparently, the first theorem on continued fractions that diverge to multiple limits 
 is that of Stern and Stolz  \cite{LorentzenWaadeland2008Book,Stern1860,Stolz1885}:

\begin{thm}[Stern-Stolz]
\label{Stern-Stolz1} Let the complex sequence $\{b_{i}\}$ satisfy $\sum
|b_{i}|$ $<\infty$. Then
\[
b_0+\kfrac{1}{b_1}\kfrac{1}{b_2}\kfrac{1}{b_3}\lowfrac{\cdots}
\]
diverges. In fact, for $p\in\{0,1\}$,
$\displaystyle \lim_{n \to \infty}A_{2n+p}=C_{p} \in \mathbb{C}$, and 
$\displaystyle \lim_{n \to \infty}B_{2n+p}=D_{p} \in \mathbb{C}$. 
\end{thm}

The proof of the Stern-Stolz Theorem goes over into the formal power series setting and the conclusion
is that limiting formal power series exist for the limits described in the theorem: inspection
of the recurrence
$A_k=b_kA_{k-1}+A_{k-2}$
shows that it converges for $k$ in the residue classes modulo ${2}$, and the same is true
of the sequence $B_k$, since it satisfies the same recurrence. That the limits are distinct
follows from the determinant formula $A_kB_{k-1}-A_{k-1}B_k=(-1)^{k+1}$.

Bowman and McLaughlin \cite{Bowman2007} established the following result on continued fractions
which diverge to three limits as an example of a more general theorem on continued fractions which diverge to any finite number of limits. 

Let $K$ be defined to be the following general continued fraction 
\begin{equation}\label{Kcf}
K:=b_0+\kfrac{-1+a_1}{1+b_1}\kfrac{-1+a_2}{1+b_2}\kfrac{-1+a_3}{1+b_3}\lowfrac{\cdots}.
\end{equation}
Because we will be interested in giving a combinatorial description for the terms of the continuants of $K$, we designate its classical numerators
and denominators by $P_k$ and $Q_k$, respectively, to distinguish them from the corresponding polynomials associated with
\eqref{eqn:CfracNotation}. With this notation, the result from \cite{Bowman2007} of interest is
the following theorem.

\begin{thm}[Example 1i from \cite{Bowman2007}]\label{SS} Let the complex sequences ${a_i}$ and ${b_i}$ satisfy
$a_i\neq 1$ for $i\geq 1$, and $\sum|a_i|+|b_i|<\infty$. For $j=1,2,3$,
\begin{align}
&\lim_{n\to\infty} P_{6n+j} = - \lim_{n\to\infty} P_{6n+j+3} = C_j\neq\infty,\\
&\lim_{n\to\infty} Q_{6n+j} = - \lim_{n\to\infty} Q_{6n+j+3} = D_j\neq\infty.
\end{align}
In fact, for $j\in\{1,2,3\}$, $K$ diverges to three limits given  by 
$$\lim_{\substack{k\to\infty\\k\equiv j \left(\bmod 3\right)}} \frac{P_k}{Q_k}. $$ 
\end{thm}

Our main result, Theorem \ref{thm:EMAnalogue}, which
gives a combinatorial description for the terms of the continuants of \eqref{Kcf}, 
shows the existence of the limits 
$C_j$ and $D_j$ as formal power series. (This can also be seen directly from \eqref{eqn:TPRelation}
and \eqref{eqn:TRecurrence} below.)

\subsection{Partition Applications}

Putting $a_i=q^i$ in Theorem \ref{EMS} gives that the Rogers-Ramanujan integer partition identities,

\begin{quote}
\emph{The number of partitions of $n$ into parts with minimal difference two equals the number of partitions of $n$ into parts congruent to $1$ or $4$ modulo $5$.}

\emph{The number of partitions of $n$ into parts greater than $1$ with minimal difference two equals the number of partitions of $n$ into parts congruent to $2$ or $3$ modulo $5$.}

\end{quote}
are equivalent to the single identity,

\begin{equation*}
1+\kfrac{q}{1}\kfrac{q^2}{1}\kfrac{q^3}{1}\kfrac{q^4}{1}\lowfrac{\cdots}
\circeq\dfrac{\left(\displaystyle{\prod_{j=1}^\infty} \dfrac{1}{(1-q^{5j+1})(1-q^{5j+4})}\right)}
{\left(\displaystyle{\prod_{j=1}^\infty} \dfrac{1}{(1-q^{5j+2})(1-q^{5j+3})}\right)},
\label{eqn:RRCF}
\end{equation*} 
where $\circeq$ indicates that the limiting classical numerator and denominator of the continued fraction on the left are equal as formal power series in $q$ to the
numerator and denominator on the right.

Thus, a combinatorial description for the terms of the continuants of continued fraction $K$, in the case where $a_i=0$,
will give a partition interpretation to the limiting classical numerator and denominator (in residue classes modulo $6$) of Ramanujan's amazing continued fraction with $3$ limits 
\cite{AndrewsBerndt2005Book,AndrewsBerndt2005}:
\begin{equation}
\begin{split}
\lim_{k \to \infty}\,\, \kfrac{1}{1}\kfrac{-1}{1+q}\kfrac{-1}{1+q^2}\lowfrac{\cdots+}&\frac{-1}{1+q^{3k+j}} \\
&=\omega^2 \left( \frac{\Omega-\omega^{j+1}}{\Omega-\omega^{j-1}}\right) \prod_{m=0}^{\infty} \frac{(1-q^{3m+2})}{(1-q^{3m+1})},
\end{split}
\label{eqn:R3LimCF}
\end{equation}
where $\omega=e^{2 \pi i /3}$, $j\in\{0,1,2\}$, and
\begin{equation*}
\Omega = \prod_{p=1}^\infty \frac{(1-\omega^2q^p)}{(1-\omega q^p)}.
\end{equation*}
It follows that when the corresponding products on the right hand have been given interpretations as partition generating functions, one obtains partition identities
which are equivalent (via the description of terms for $K$'s continuants) to Ramanujan's three-limit continued fraction.
This will be attained in a sequel, and was one of the chief motivations for the present paper. 

To state the problem solved in this paper most succinctly,  we \emph{give a combinatorial description for the terms of the polynomials} $P_k$,
\emph{defined recursively in the non-commutative indeterminates} $a_i$ \emph{and} $b_i$ \emph{by}:

\[P_k=(-1+a_k)P_{k-2}+(1+b_k)P_{k-1},\]
\emph{with initial conditions} $P_0=b_0$ \emph{and} $P_1=-1+b_0 + a_1 + b_1 b_0$.

\subsection{Results} \label{overview:Results}

This paper studies a number of new and interrelated  
sequences of polynomials whose terms are described combinatorially. 
These sequences of polynomials are of two types. The first arise from
the classical Euler-Minding Theorem; they exhibit a modulo two or four behavior as a function of their index. 
The second arise from the sequence $P_k$; these exhibit a modulo six
behavior. The terms of $P_k$ are characterized by Theorem \ref{thm:EMAnalogue}, which is
the main result of this paper. 
Equalities are induced between the two types because the continued fraction
\eqref{eqn:CfracNotation}  can be transformed into \eqref{Kcf} by making the change of variables
$a_i\mapsto -1+a_i$ and $b_i\mapsto 1+b_i$, for $i\geq 1$. This results
 in non-trivial identities, since the sum in the non-commutative version of the Euler-Minding 
Theorem (see Section \ref{subsec:nccf})
now has intensive sieving occurring, while the polynomials on the other side are expressed in terms of their monomials. Important special cases arise when either the variables $a_i$ or $b_i$ vanish. For
the continued fraction $K$, this results in the polynomial sequences $C_k$, $D_k$, $G_k$, and $H_k$
introduced in Section \ref{main}. Section \ref{sec:last} examines the resulting polynomial identities and
also gives applications to common second order linear recurrence sequences of integers. 
In a future paper we will apply Corollary \ref{cor:DenSeqEMAnalogue} of Theorem \ref{thm:EMAnalogue} 
to find integer partition identities
equivalent to \eqref{eqn:R3LimCF}.

The simplest example of our results is perhaps the following, which comes from Corollaries  \ref{cor:ck} and \ref{last}:
\begin{equation}
-\dfrac{2 \sqrt{3}}{3} \mathrm{Im} \left( e^{k\pi i/3} \right)=\sum_{k \geq \lambda_1 >^2 \lambda_2 >^2 \dots >^2 \lambda_\ell = 1} (-1)^\ell=
-\chi_1(k)+\sum_{\bs{\lambda} \in \bs{D}_k} (-1)^{\frac{k-\ell+1}{2}},
\label{eqn:res_example}
\end{equation}
where $\chi_1(k)$ is the nonprincipal Dirichlet character modulo 4, and $\bs{D}_k$ is the set of finite integer sequences (depending on $k$) satisfying,
\begin{description}
\item[D1] $k \geq \lambda_1 > \lambda_2 > \dots > \lambda_\ell \geq 2$.
\item[D2] $\lambda_1 \equiv k (\bmod 2)$.
\item[D3] $\lambda_j \not\equiv \lambda_{j-1} (\bmod 2)$.
\item[D4] $\lambda_\ell\equiv 0(\bmod 2)$.
\end{description}
The first expression in \eqref{eqn:res_example} indicates a six-fold pattern in
the integer sequences given by the sums, although from superficial appearances of the sums,
 one might expect a two-fold or four-fold pattern. The
interpretation of the first equality is beautiful and surprising: 
\begin{quotation}
\emph{Let $\bs{C}_k$ denote the set of increasing sequences of positive integers of minimal difference two, with first term $1$ and largest term less than or equal to $k$. Then the number of elements of $\bs{C}_k$ of even length minus the number of elements of odd length is given by the six-periodic integer sequence $0,-1,-1,0,1,1,\dots$, where the first element of the sequence is indexed by $k=0$.}
\end{quotation}

In Section \ref{subsec:relating}  we  give a simple proof of this result which is
independent of the more general theory developed in this paper.

Finally, when a decreasing sequence $\lambda_i$ satisfies condition \textbf{D3} above, we say
that it is an \emph{alternating parity sequence}. Partitions formed from sequences of such parts have been studied by Andrews
\cite{Andrews1984,Andrews2010}. It is easy to show that these kinds of partitions arise naturally from Euler's combinatorial description of the continuants of \eqref{eqn:CfracNotation} in the case $a_i=1$ and $b_i=q^i$. In Section \ref{main}  \emph{alternating triality sequences} arise, which
are similar, except the congruence conditions on the successive terms are modulo three, instead of two.

\section{Preliminaries and Lemmas}
\label{sec:EMAnalogue}

\subsection{Continued Fractions with Noncommuting indeterminates}\label{subsec:nccf}

The fundamental recurrence formulas for the classical numerators and denominators of continued fractions
are used for typical proofs of the Euler-Minding Theorem and they are used to prove 
Theorem \ref{thm:EMAnalogue}.   These recurrences state that for $k\geq 1$, 
\begin{equation}
A_k=a_kA_{k-2}+b_kA_{k-1},
\label{eqn:ARecur}
\end{equation}
and
\begin{equation}
B_k=a_kB_{k-2}+b_kB_{k-1},
\label{eqn:BRecur}
\end{equation}
where $A_{-1}=1$ and $B_{-1}=0$.  Recurrence formulas with left or right multiplication by noncommuting indeterminates have been considered since at least 1913 \cite{Wedderburn1913}. The convention of writing parts of partitions in descending order motivates us to consider recurrences \eqref{eqn:ARecur} and \eqref{eqn:BRecur} with noncommuting indeterminates. In this context we speak of the continued fraction \eqref{eqn:CfracNotation} as
having noncommuting indeterminates; we define the \emph{classical numerators} and \emph{denominators} as the respective sequences of polynomials in noncommutative indeterminates satisfying equations \eqref{eqn:ARecur} and \eqref{eqn:BRecur}, with initial conditions $A_0=b_0$, $A_1=b_1b_0+a_1$, $B_0=1$, and $B_1=b_1$.  Each classical numerator, $A_k$, and classical denominator $B_k$ is an element of the monoid ring $\mathbb{Z}[\mathcal{M}]$, where $\mathcal{M}$ is the monoid generated by $\{a_{j+1},b_j\}_{j \geq 0}$ with identity $\epsilon$.  The integers are isomorphic to the subring 
$\mathbb{Z}\epsilon$ of $\mathbb{Z}[\mathcal{M}]$; we abuse $1\epsilon$, as usual, by writing it simply
as $1$.  The product in $\mathcal{M}$ is denoted by concatenation.  Definition \ref{def:MonoidRingStuff} provides terminology and notation for $\mathbb{Z}[\mathcal{M}]$ and its elements.

\begin{definition}
We call the elements $P$ of $\mathbb{Z}[\mathcal{M}]$ \emph{polynomials}.  We write $P$ in the form
\begin{equation}
P=\sum_{m\in \mathcal{M}} c_{m} m,
\label{eqn:ExpressP}
\end{equation}
where $c_m \in \mathbb{Z}$ and all but finitely many $c_m$ are zero.  The \emph{support of $P$}, denoted by $\mathrm{supp}(P)$, is the set 
\[\mathrm{supp}(P)=\{ m \in \mathcal{M} : c_m \neq 0\}.\]
We write 
\begin{equation}
P=\sum_{m \in \mathrm{supp}(P)} c_m m
\label{eqn:ExpressPSupp}
\end{equation}
to keep polynomial sums finite.  We call the elements of $\mathrm{supp}(P)$ the \emph{monomials of $P$}, and for a monomial $m$ of $P$, we call $c_m m$ a \emph{term} of $P$.  So here, monomials do not have integer coefficients, while terms do.  We call the coefficient of the identity $\epsilon$ in \eqref{eqn:ExpressP} (not \eqref{eqn:ExpressPSupp}, since it may be that $\epsilon\notin\mathrm{supp}(P))$
 the \emph{constant} of $P$.  Thus the constant of $P$ can be zero.  
\label{def:MonoidRingStuff}
\end{definition}

Since the goal is to give combinatorial descriptions for the terms of classical numerator and denominator polynomials of $K$, we employ vectors whose components are indices of the elements of the support of these polynomials.  In the sequel and throughout, we display the components of an $\ell$-dimensional vector $\bs{\lambda}$ as $[\lambda_1, \lambda_2, \dots,\lambda_\ell]$.  Definition \ref{def:Indices} below defines vectors directly related to the monomials of a given $P \in \mathbb{Z}[\mathcal{M}]$.  For the definition, we use the noncommutative product notation inductively
 defined for $n\geq 1$ by
\[\prod_{j=1}^nd_i=d_1\prod_{j=1}^{n-1} d_{j+1},\]
and the empty product is $\epsilon$ as usual.

\begin{definition}
Let $m$ be a monomial of $P \in \mathbb{Z}[\mathcal{M}]$,
\[m=\prod_{j=1}^\ell y_j,\]
where $y_j \in \{a_{i+1},b_i\}_{i \geq 0}$.  We denote the \emph{degree} or \emph{length} of the monomial $m$ by $\ell=\ell(m)$; we usually suppress the dependence of $\ell$ on $m$.
\begin{enumerate}[(i)]
\item The \emph{index of $m$} is the vector $\bs{\lambda}(m)=[\lambda_1,\lambda_2,\dots,\lambda_\ell]$, where $y_j=a_{u}$ implies $\lambda_j=u$ and $y_j=b_{u}$ implies $\lambda_j=u$.  
\item The \emph{$a$-index} of $m$ is the vector $\bs{\alpha}(m)=[\alpha_1,\alpha_2,\dots,\alpha_\ell]$, where  
\[
\alpha_j=\begin{cases}
u & \text{if} \quad y_j=a_u,\\
0 & \text{otherwise}.
\end{cases}
\]
\item The \emph{$b$-index} of $m$ is the vector $\bs{\beta}(m)=[\beta_1,\beta_2,\dots,\beta_\ell]$, where  
\[
\beta_j=\begin{cases}
u & \text{if} \quad y_j=b_u,\\
0 & \text{otherwise}.
\end{cases}
\]

\end{enumerate}
Note that for a monomial $m$ the index of $m$ is the sum of the $a$-index and $b$-index: 
$\bs{\lambda}(m)=\bs{\alpha}(m)+\bs{\beta}(m).$
\label{def:Indices}
\end{definition}

\begin{example}
The monomial $a_6b_4b_3b_2a_1$ has index $[6,4,3,2,1]$. It has $a$-index  $[6,0,0,0,1]$ and $b$-index  $[0,4,3,2,0]$.  Monomial $b_5a_4b_2 b_0$ has $a$-index $[0,4,0,0]$,  $b$-index  $[5,0,2,0]$, and index $[5,4,2,0]$.
\label{ex:Indices}
\end{example}

By a \emph{formal power series} we mean an element of the monoid ring $\mathbb{Z}[[\mathcal{M}]]$, 
that  is,
an expression of the form
\[
c=\sum_{m\in \mathcal{M}} c_{m} m,
\]
where now we do not require all but finitely many $c_{m}$ to be $0$. Addition and multiplication 
are defined as usual.

Before studying $K$ we derive the noncommutative description of the terms of the continuants of the general continued
fraction \eqref{eqn:CfracNotation}.

\subsection{A Noncommutative Euler-Minding Theroem}

Minding \cite{Minding1869} seems to have been the first to give the following slightly more general version of Euler's result \cite{Euler1764}.  See also \cite{Peron1954Book}.
\begin{thm}[Euler-Minding Theorem]
The classical numerators and denominators of the continued fraction 
\begin{equation}
b_0+\kfrac{a_1}{b_1}\kfrac{a_2}{b_2}\kfrac{a_3}{b_3}\lowfrac{\cdots}
\label{eqn:EMCF}
\end{equation}
in commutative indeterminates $\{a_{j+1},b_j\}_{j \geq 0}$ are given by
\begin{equation}
A_k=b_kb_{k-1}\cdots b_1b_0 \left[ 1 + \sum_{1\leq h_j <^2 h_{j-1} <^2 \cdots <^2 h_1 \leq k} \dfrac{a_{h_1} a_{h_2} \cdots a_{h_j}}{b_{h_1}b_{h_1-1} b_{h_2} b_{h_2-1} \cdots b_{h_j} b_{h_j-1}} \right],
\label{eqn:EMNum}
\end{equation}
and
\begin{equation}
B_k=b_kb_{k-1}\cdots b_1 \left[ 1 + \sum_{2\leq h_j <^2 h_{j-1} <^2 \cdots <^2 h_1 \leq k} \dfrac{a_{h_1} a_{h_2} \cdots a_{h_j}}{b_{h_1}b_{h_1-1} b_{h_2} b_{h_2-1} \cdots b_{h_j} b_{h_j-1}} \right].
\label{eqn:EMDen}
\end{equation}
\label{thm:EM}
\end{thm}

Note that this theorem does not immediately give a description for the terms for each continunant since the terms are rational,
not monomial. But this is easy to remedy. 

Theorem \ref{thm:EM} expresses $A_k$ and $B_k$ as rational functions in commuting indeterminates. One 
obtains the noncommutative version by multiplying through by the $b$-product in front, canceling, and 
then ordering the terms so that the indices from left to right are decreasing; the construction of the terms in the sum 
guarantees that the indices are distinct, so no ambiguity between, say $a_ib_i$ 
and $b_ia_i$ can occur.   
For the classical numerators, \eqref{eqn:ARecur} must be satisfied along with the initial conditions $A_0=b_0$ and $A_1=b_1b_0+a_1$.  Induction on \eqref{eqn:ARecur} gives that $A_k$ is a polynomial in the indeterminates $\{a_{j+1},b_j\}_{j=0}^k$.  Since \eqref{eqn:ARecur} introduces the new indeterminates $a_k$ and $b_k$ by left
multiplication,  the indices of the terms of the classical numerators are in descending order.  Therefore, the result of expanding each summand of \eqref{eqn:EMNum} and putting the indices into descending order satisfies \eqref{eqn:ARecur} with noncommuting indeterminates.  
Thus,
\begin{equation}\label{eqn:EMNumPoly}
A_k=\prod_{t=0}^{k} b_{k-t}+
\sum_{\substack{1\leq h_j <^2 h_{j-1} <^2 \cdots <^2 h_1 \leq k\\j\geq 1}} \ \prod_{t=0}^{k-h_1-1} b_{k-t} \times \prod_{u=1}^{j} \left( a_{h_u}  \prod_{v=2}^{h_u-h_{u+1}-1} b_{h_u-v} \right).
\end{equation}
A summand appearing in the second term of \eqref{eqn:EMNumPoly} has the form
\begin{multline*}
b_kb_{k-1}\cdots b_{h_1+1}\times (a_{h_1}b_{h_1-2}b_{h_1-3}\cdots b_{h_2+1})(a_{h_2}b_{h_2-2}\cdots b_{h_3+1}) \\
\cdots (a_{h_j} b_{h_j-2} \cdots b_0)  .
\end{multline*}
Observe that the largest index is $k$ and the indices are distinct nonnegative integers.  When an 
$a$-index is equal to some $h_j$, the next index is $h_j-2$, since the next index is either $b$-index $h_j-2$ or $a$-index $h_{j+1}=h_j-2$.  When the index is some $b$-index $h_i-s$, the next index is $h_i-s-1$, since the next index is either $a$-index $h_{i-1}=h_i-s-1$ or $b$-index $h_i-s-1$.  Finally, the last index is either zero or one.  The last index is a $b$-index zero when $h_j>1$, and it is the $a$-index $1$ when $h_j=1$.  

It is now easy to describe the subset of monomials of $\mathcal{M}$ occurring in the noncommutative  Euler-Minding Theorem: let $\mathcal{A}_k$ be the set of monomials with $a$-index $\bs{\alpha}$, $b$-index $\bs{\beta}$, and index $\bs{\lambda}=\bs{\alpha}+\bs{\beta}$ satisfying the following properties.
\begin{description} 
\item[A1] $k = \lambda_1 > \lambda_2 > \dots > \lambda_\ell \geq 0$.
\item[A2] If $\lambda_j=\alpha_j$, then $\lambda_{j+1} = \lambda_{j}-2$.
\item[A3] If $\lambda_j=\beta_j$, then $\lambda_{j+1}=\lambda_{j}-1$.
\item[A4] Either $\lambda_\ell=\beta_\ell=0$ or $\lambda_\ell=\alpha_\ell=1$.
\end{description}
It is clear that {\bf A1}--{\bf A4} describe the terms of \eqref{eqn:EMNumPoly}.

For example, $\mathcal{A}_0=\{b_0\}$, and $\mathcal{A}_1=\{b_1b_0,a_1\}$.  Indeed, the index of any element of $\mathcal{A}_0$ has $\lambda_1=0$ by {\bf A1}.  The only possible $a$ and $b$ indices are each $[0]$.  These vectors satisfy {\bf A1}--{\bf A4}, so $\mathcal{A}_0=\{b_0\}$.  Also $\mathcal{A}_1=\{b_1b_0,a_1\}$; the index of any element of $\mathcal{A}_1$ has $\lambda_1=1$ by {\bf A1}.  So, the possible indices are $[1,0]$ and $[1]$.  By {\bf A2} the vector $[1,0]$ cannot be an $a$-index.  The monomial $b_1b_0$ with $a$-index $[0,0]$ and $b$-index $[1,0]$ satisfies {\bf A1}--{\bf A4}.  Thus, $b_1b_0$ is in $\mathcal{A}_1$.  By {\bf A4} the vector $[1]$ is not a $b$ index.  The monomial $a_1$ with $a$-index $[1]$ and $b$-index $[0]$ satisfies {\bf A1}--{\bf A4}.  Thus, $a_1$ is in $\mathcal{A}_1$, and $\mathcal{A}_1=\{b_1b_0,a_1\}$.

It is not hard to show that $b_0$ is a term of $A_k$ if and only if $k$ is even and that $a_1$ is 
a term of $A_k$ if and only if $k$ is odd.  Further it can be shown, although we don't take it up
here, that $\lim_{k\to\infty} A_{2k}$ and $\lim_{k\to\infty} A_{2k+1}$ exist and are distinct in
$\mathbb{Z}[[\mathcal{M}]]$.

\begin{thm}[Noncommutative Euler-Minding Theorem]
The classical numerators of the continued fraction 
\[b_0+\kfrac{a_1}{b_1}\kfrac{a_2}{b_2}\kfrac{a_3}{b_3}\lowfrac{\cdots}\]
in noncommutative indeterminates $\{a_{j+1},b_j\}_{j \geq 0}$ for $k\geq 0$ are given by
\begin{equation}
A_k= \sum_{m \in \mathcal{A}_k} m.
\label{eqn:NoncomEMNum}
\end{equation}
\label{thm:NoncomEM}
\end{thm}
\begin{proof}
As explained in the paragraph following Theorem \ref{thm:EM}, ordering the resulting subscripts in descending order and canceling the $b_j$s in \eqref{eqn:EMNum} gives \eqref{eqn:NoncomEMNum}.  
\qed
\end{proof}

\subsection{Lemmas}

Let $P_k(a_1,a_2,\dots,a_k;b_0,b_1,b_2,\dots,b_k)$ and $Q_k(a_2,\dots,a_k;b_1,b_2,\dots,b_k)$ be the $k$th classical numerators and denominators of the continued fraction
\[K=b_0+\kfrac{-1+a_1}{1+b_1}\kfrac{-1+a_2}{1+b_2}\kfrac{-1+a_3}{1+b_3}\lowfrac{\cdots},\]
where indeterminates $\{a_{j+1},b_j\}_{j\geq 0}$ are noncommutative.  By  the fundamental recurrence formulas \eqref{eqn:ARecur} and \eqref{eqn:BRecur}, the classical numerators and denominators of $K$ satisfy
\[X_k=(-1+a_k)X_{k-2}+(1+b_k)X_{k-1},\]
with initial conditions $P_0=b_0$, $Q_0=1$, $P_1=-1+b_0+b_1b_0+a_1$, and $Q_1=1+b_1$. The first three classical numerators are:
\begin{align*}
P_0&=b_0, \\
P_1&=-1+b_0+b_1b_0+a_1, \\
P_2&=-1+a_2b_0+b_1b_0+a_1-b_2+b_2b_0+b_2b_1b_0+b_2a_1 .
\end{align*}
The following lemma gives a relationship between the $k$th classical denominator and $k+1$th classical numerator.  

\begin{lem} 
\begin{equation}
Q_k=-P_{k+1}(0,a_1,a_2,\dots,a_{k};0,0,b_1,\dots,b_{k}).
\label{eqn:Num2Den}
\end{equation}
\label{lem:Num2Den}
\end{lem}
\begin{proof} Let $x_k$ denote the right hand side of \eqref{eqn:Num2Den}.  Then $x_0=1$ and $x_1=1+b_1$.  Observe that $x_k$ satisfies $x_k=(-1+a_k)x_{k-2}+(1+b_k)x_{k-1}$.  This is the same recurrence and initial conditions satisfied by $Q_k$.
\qed
\end{proof}

Define the sequence of polynomials $R_k$ as follows: set $R_{-1}=0$ and for $k\geq 0$, let
\begin{equation}
\label{eqn:TDefn}
R_k(a_1,a_2,\dots,a_k;b_0,b_1,\dots,b_k)=P_k-R_{k-1}(a_1,a_2,\dots,a_{k-1};b_0,b_1,\dots,b_{k-1}),
\end{equation}
so that 
\begin{equation}
\label{eqn:TPRelation}
P_k=R_k+R_{k-1}.
\end{equation}

The classical recurrence formula for $P_k$,
\begin{equation}
P_k=(-1+a_k)P_{k-2}+(1+b_k)P_{k-1},
\label{eqn:ClassRecP}
\end{equation}
and \eqref{eqn:TPRelation} give a recurrence formula for $R_k$,
\begin{equation}
R_k=-R_{k-3}+a_k(R_{k-2}+R_{k-3})+b_k(R_{k-1}+R_{k-2}).
\label{eqn:TRecurrence}
\end{equation}
For consistency, set $a_0=0$ and initialize $R_{-3}=0$, $R_{-2}=1$, and $R_{-1}=0$. Interpreting this recurrence formula is the key to our proof of  Theorem \ref{thm:EMAnalogue}.  

For future reference the first seven elements in the sequence $\{R_n\}_{n=0}^\infty$ are listed:

\begin{subequations}
{\allowdisplaybreaks
\begin{flalign}
R_{0}=&\,b_0, 
\displaybreak[2]
\\
R_{1}=&-1+a_1+b_1b_0,
\displaybreak[2]
\\
R_{2}=&\,a_2b_0-b_2+b_2a_1+b_2b_1b_0+b_2b_0,
\displaybreak[2]
\\
R_{3}=&-b_0-a_3+a_3a_1+a_3b_1b_0+a_3b_0+b_3a_2b_0-b_3b_2+b_3b_2a_1\nonumber\\
&+b_3b_2b_1b_0+b_3b_2b_0-b_3+b_3a_1+b_3b_1b_0,
\displaybreak[2]
\\
R_{4}=&\,1-a_1-b_1b_0+a_4a_2b_0-a_4b_2+a_4b_2a_1+a_4b_2b_1b_0+a_4b_2b_0-a_4\nonumber\\
&+a_4a_1+a_4b_1b_0-b_4b_0-b_4a_3+b_4a_3a_1+b_4a_3b_1b_0+b_4a_3b_0\nonumber\\
&+b_4b_3a_2b_0-b_4b_3b_2+b_4b_3b_2a_1+b_4b_3b_2b_1b_0+b_4b_3b_2b_0-b_4b_3\nonumber\\
&+b_4b_3a_1+b_4b_3b_1b_0+b_4a_2b_0-b_4b_2+b_4b_2a_1+b_4b_2b_1b_0+b_4b_2b_0,
\displaybreak[2]
\\
R_5=&-a_2b_0+b_2-b_2a_1-b_2b_1b_0-b_2b_0-a_5b_0-a_5a_3+a_5a_3a_1\nonumber\nonumber\\
&+a_5a_3b_1b_0+a_5a_3b_0+a_5b_3a_2b_0-a_5b_3b_2+a_5b_3b_2a_1+a_5b_3b_2b_1b_0\nonumber\\
&+a_5b_3b_2b_0-a_5b_3+a_5b_3a_1+a_5b_3b_1b_0+a_5a_2b_0-a_5b_2+a_5b_2a_1\nonumber\\
&+a_5b_2b_1b_0+a_5b_2b_0+b_5-b_5a_1-b_5b_1b_0+b_5a_4a_2b_0-b_5a_4b_2\nonumber\\
&+b_5a_4b_2a_1+b_5a_4b_2b_1b_0+b_5a_4b_2 b_0-b_5a_4+b_5a_4a_1+b_5a_4b_1b_0\nonumber\\
&-b_5b_4b_0-b_5b_4a_3+b_5b_4a_3a_1+b_5b_4a_3b_1b_0+b_5b_4a_3b_0+b_5b_4b_3a_2b_0\nonumber\\
&-b_5b_4b_3b_2+b_5b_4b_3b_2a_1+b_5b_4b_3b_2b_1b_0+b_5b_4b_3b_2b_0-b_5b_4b_3\nonumber\\
&+b_5b_4b_3a_1+b_5b_4b_3b_1b_0+b_5b_4a_2b_0-b_5b_4b_2+b_5b_4b_2a_1+b_5b_4b_2b_1b_0\nonumber\\
&+b_5b_4b_2b_0-b_5b_0-b_5a_3+b_5a_3a_1+b_5a_3b_1b_0+b_5a_3b_0+b_5b_3a_2b_0\nonumber\\
&-b_5b_3b_2+b_5b_3b_2a_1+b_5b_3b_2b_1b_0+b_5b_3b_2b_0-b_5b_3+b_5b_3a_1\nonumber\\
&+b_5b_3b_1b_0 ,
\\
\text{and}\nonumber
\displaybreak[2]
\\
R_6=&
\,b_0+a_3-a_3a_1-a_3b_1b_0-a_3b_0-b_3a_2b_0+b_3b_2-b_3b_2a_1-b_3b_2b_1b_0\nonumber\\
&-b_3b_2b_0+b_3-b_3a_1-b_3b_1b_0
+a_6-a_6a_1-b_1b_0+a_6a_4a_2b_0\nonumber\\
&-a_6a_4b_2+a_6a_4b_2a_1+a_6a_4b_2b_1b_0+a_6a_4b_2b_0-a_6a_4+a_4a_1\nonumber\\
&+a_6a_4b_1b_0-a_6b_4b_0-a_6b_4a_3+a_6b_4a_3a_1+a_6b_4a_3b_1b_0+a_6b_4a_3b_0\nonumber\\
&+a_6b_4b_3a_2b_0-a_6b_4b_3b_2+a_6b_4b_3b_2a_1+a_6b_4b_3b_2b_1b_0+a_6b_4b_3b_2b_0\nonumber\\
&-a_6b_4b_3+a_6b_4b_3a_1+a_6b_4b_3b_1b_0+a_6b_4a_2b_0-a_6b_4b_2+a_6b_4b_2a_1\nonumber\\
&+a_6b_4b_2b_1b_0+a_6b_4b_2b_0-a_6b_0-a_6a_3+a_6a_3a_1+a_6a_3b_1b_0+a_6a_3b_0\nonumber\\
&+a_6b_3a_2b_0-a_6b_3b_2+a_6b_3b_2a_1+a_6b_3b_2b_1b_0+a_6b_3b_2b_0-a_6b_3\nonumber\\
&+a_6b_3a_1+a_6b_3b_1b_0
-b_6a_2b_0+b_6b_2-b_6b_2a_1-b_6b_2b_1b_0-b_6b_2b_0\nonumber\\
&-b_6a_5b_0-b_6a_5a_3+b_6a_5a_3a_1+b_6a_5a_3b_1b_0+b_6a_5a_3b_0+b_6a_5b_3a_2b_0\nonumber\\
&-b_6a_5b_3b_2+b_6a_5b_3b_2a_1+b_6a_5b_3b_2b_1b_0+b_6a_5b_3b_2b_0-b_6a_5b_3\nonumber\\
&+b_6a_5b_3a_1+b_6a_5b_3b_1b_0+b_6a_5a_2b_0-b_6a_5b_2+b_6a_5b_2a_1+b_6a_5b_2b_1b_0\nonumber\\
&+b_6a_5b_2b_0+b_6b_5-b_6b_5a_1-b_6b_5b_1b_0+b_6b_5a_4a_2b_0-b_6b_5a_4b_2\nonumber\\
&+b_6b_5a_4b_2a_1+b_6b_5a_4b_2b_1b_0+b_6b_5a_4b_2 b_0-b_6b_5a_4+b_6b_5a_4a_1\nonumber\\
&+b_6b_5a_4b_1b_0-b_6b_5b_4b_0-b_6b_5b_4a_3+b_6b_5b_4a_3a_1\nonumber\\
&+b_6b_5b_4a_3b_1b_0+b_6b_5b_4a_3b_0+b_6b_5b_4b_3a_2b_0-b_6b_5b_4b_3b_2\nonumber\\
&+b_6b_5b_4b_3b_2a_1+b_6b_5b_4b_3b_2b_1b_0+b_6b_5b_4b_3b_2b_0-b_6b_5b_4b_3\nonumber\\
&+b_6b_5b_4b_3a_1+b_6b_5b_4b_3b_1b_0+b_6b_5b_4a_2b_0-b_6b_5b_4b_2+b_6b_5b_4b_2a_1\nonumber\\
&+b_6b_5b_4b_2b_1b_0+b_6b_5b_4b_2b_0-b_6b_5b_0-b_6b_5a_3+b_6b_5a_3a_1\nonumber\\
&+b_6b_5a_3b_1b_0+b_6b_5a_3b_0+b_6b_5b_3a_2b_0-b_6b_5b_3b_2+b_6b_5b_3b_2a_1\nonumber\\
&+b_6b_5b_3b_2b_1b_0+b_6b_5b_3b_2b_0-b_6b_5b_3+b_6b_5b_3a_1+b_6b_5b_3b_1b_0
+b_6\nonumber\\
&-b_6a_1-b_6b_1b_0+b_6a_4a_2b_0-b_6a_4b_2+b_6a_4b_2a_1+b_6a_4b_2b_1b_0\nonumber\\
&+b_6a_4b_2b_0-b_6a_4+b_6a_4a_1+b_6a_4b_1b_0-b_6b_4b_0-b_6b_4a_3+b_6b_4a_3a_1\nonumber\\
&+b_6b_4a_3b_1b_0+b_6b_4a_3b_0+b_6b_4b_3a_2b_0-b_6b_4b_3b_2+b_6b_4b_3b_2a_1\nonumber\\
&+b_6b_4b_3b_2b_1b_0+b_6b_4b_3b_2b_0-b_6b_4b_3+b_6b_4b_3a_1+b_6b_4b_3b_1b_0\nonumber\\
&+b_6b_4a_2b_0-b_6b_4b_2+b_6b_4b_2a_1+b_6b_4b_2b_1b_0+b_6b_4b_2b_0.
\end{flalign}
\label{eqn:ManyRs}
}
\end{subequations}

\begin{lem}
For $k\geq 2$, the polynomials $R_k$, $R_{k-1}$, and $R_{k-2}$ have pairwise disjoint supports; there is no cancellation of terms in the sum $R_k+R_{k-1}+R_{k-2}$.

\label{lem:TDisjointTerms}
\end{lem}

\begin{proof}
This follows easily by induction on recurrence formula \eqref{eqn:TRecurrence}.
\qed
\end{proof}

\begin{cor}
For $k\geq 0$, let $r_k$ count the number of terms of $R_k$.  The sequence of integers $\{r_k\}_{k=0}^\infty$ satisfies the recurrence formula
\[
\begin{cases} r_0=1, \quad r_1=3, \quad r_2=5\\
r_k=r_{k-1}+2r_{k-2}+2r_{k-3},\end{cases}
\]
and has generating function
\[
\sum_{k\geq 0}r_kx^k=\frac{1+2x}{1-x-2x^2-2x^3}.
\]
\label{cor:TNumTerms}
\end{cor}
\begin{proof}
This is immediate from Lemma \ref{lem:TDisjointTerms} and \eqref{eqn:TRecurrence}. The calculation
of the generating function follows by the usual method.
\qed
\end{proof}

\begin{lem}
Let $T$ be a term of $R_k$.  For $j>0$:
\begin{enumerate}
\item The degree of $T$ in each variable $a_1$, $a_2$, $\dots$, $a_k$, $b_0$, $b_1$, $\dots$, $b_k$ is at most one.
\item If $a_j$ is a factor of $T$, then $b_j$ is not a factor of $T$.
\end{enumerate}

\label{lem:TDisjointSubscripts}
\end{lem}
\begin{proof}
By induction these statements are true for the terms of $P_k$ by \eqref{eqn:ClassRecP}.  The result for $R_k$ then follows from \eqref{eqn:TPRelation}.  
\qed
\end{proof}

Let $\rho(k)$  be the periodic sequence:
\begin{equation*}\rho(k)=\begin{cases} -1 & \text{ if } k\equiv 1 (\bmod 6) \; \\
1 & \text{ if } k\equiv 4 (\bmod 6) \; \\
0 & \text{otherwise }\hspace*{\fill}.
\end{cases}\end{equation*}
Observe that the constant of $R_k$ equals $\rho(k)$ for $k=0,1,\dots,5$.  Further observe that the coefficient of each term in $R_k$ is $\pm 1$, for $k=0,1,\dots,5$.  More generally the following lemma holds.

\begin{lem}
The constant of each polynomial $R_k$ is $\rho(k)$.  Further, the coefficient of any term $T$ of $R_k$ is $\pm 1$.
\label{lem:CfracCoefLemma}
\end{lem}
\begin{proof}  Let $\mathrm{Const}(R_i)=R_i(0,0,\dots,0;0,0,\dots,0)$ be the constant of $R_i$.  By \eqref{eqn:TRecurrence}, $\mathrm{Const}(R_k)=-\mathrm{Const}(R_{k-3})$.  That the constant term of $R_k$ is $\rho(k)$ follows by induction.  In \eqref{eqn:ManyRs}, the coefficients of $R_0$, $R_1$, and $R_2$ are $\pm 1$.  The lemma now follows by Lemma \ref{lem:TDisjointTerms} and \eqref{eqn:TRecurrence}. 
\qed
\end{proof}

Proposition \ref{prop:TkSubscrpt} will show the following definition characterizes $\mathrm{supp}(R_k)\backslash\{\epsilon\}$.
\begin{definition}
For $k\geq 0$, define $\mathcal{R}_k$ to be the set of monomials whose index $\bs{\lambda}$, $a$-index $\bs{\alpha}$, and $b$-index $\bs{\beta}$ satisfy the following properties:
\begin{description}
	\item[R1] $k \geq \lambda_1 > \lambda_2 > \dots \lambda_\ell \geq 0$.
	\item[R2] $\lambda_1\equiv k (\bmod 3)$.
	\item[R3] If $\lambda_j = \alpha_j$, then $\lambda_j \not \equiv \lambda_{j+1}+1 (\bmod 3)$.
	\item[R4] If $\lambda_j = \beta_j$, then $\lambda_j \not \equiv \lambda_{j+1} (\bmod 3)$.
	\item[R5] If $\lambda_\ell = \alpha_\ell$, then $\lambda_\ell \not \equiv 2 (\bmod 3)$.
	\item[R6] If $\lambda_\ell = \beta_\ell$, then $\lambda_\ell \not \equiv 1 (\bmod 3)$.
\end{description}
\end{definition}

Note that property {\bf R1} implies that monomials in $\mathcal{R}_k$ satisfy the conditions of Lemma \ref{lem:TDisjointSubscripts}.  Example \ref{ex:TkExample} below shows the sets $\{b_0\}$, $\{b_1b_0,a_1\}$, and $\{b_2b_1b_0,b_2a_1,a_2b_0,b_2b_0,b_2\}$ are $\mathcal{R}_0$, $\mathcal{R}_1$, and $\mathcal{R}_2$, respectively.

\begin{example}\label{ex:FirstTsets}
Property {\bf R1} implies that all elements of $\mathcal{R}_0$ have an index with $\lambda_1=\lambda_\ell=0$.  Thus, any monomial in $\mathcal{R}_0$ has index, $a$-index, and $b$-index each equal to $[0]$.  This index, $a$-index, and $b$-index satisfy {\bf R1}--{\bf R6}, thus $\mathcal{R}_0=\{b_0\}$.  

Properties {\bf R1} and {\bf R2} imply that all elements of $\mathcal{R}_1$ have an index with $\lambda_1=1$.  Possible indices are $[1,0]$ and $[1]$.  When $\bs{\lambda}=[1,0]$, the $a$-index $[1,0]$ and $b$-index $[0,0]$ do not satisfy {\bf R3}, so $a_1b_0\notin\mathcal{R}_1$. However, the monomial with $a$-index $[0,0]$ and $b$-index $[1,0]$ satisfies  {\bf R1}--{\bf R6}.  Thus 
$b_1b_0\in\mathcal{R}_1$.  The monomials index $[1]$ with $a$-index $[1]$ and $b$-index $[0]$ satisfies  {\bf R1}--{\bf R6}, thus $a_1\in \mathcal{R}_1$.  The monomials index $[1]$ with $a$-index $[0]$ and $b$-index $[1]$ does not satisfy {\bf R6}, so $b_1\notin\mathcal{R}_1$.  Thus 
$\mathcal{R}_1=\{b_1b_0,a_1\}$.

Properties {\bf R1} and {\bf R2} imply that all elements of $\mathcal{R}_2$ have an index with $\lambda_1=2$.  Possible monomial indices are $[2,1,0]$, $[2,1]$, $[2,0]$, and $[2]$.  For a monomial in $\mathcal{R}_2$ with index $[2,1,0]$, $\alpha_1 \neq 2$ and $\alpha_1 \neq 1$ by {\bf R3}. Thus, $a_2b_1b_0$, $a_2a_1b_0$, $b_2a_1b_0\notin\mathcal{R}_2$.  However the monomial with index $[2,1,0]$, $a$-index $[0,0,0]$ and $b$-index $[2,1,0]$ does satisfy {\bf R1}--{\bf R6}.  Thus, 
$b_2b_1b_0\in \mathcal{R}_2$.  For a monomial in $\mathcal{R}_2$ with index $[2,1]$, $\alpha_1 \neq 2$,
so {\bf R3} implies that $a_2b_1,a_2a_1\notin\mathcal{R}_2$.  For a monomial with index $[2,1]$, $b_2\neq 1$, so {\bf R6} implies that $b_2b_1\notin\mathcal{R}_2$.   The monomial with $(\bs{\alpha},\bs{\beta})=([0,1],[2,0])$ satisfies {\bf R1}--{\bf R6}.  Thus, $b_2a_1 \in \mathcal{R}_2$.  For index $[2,0]$, the monomials with $(\bs{\alpha},\bs{\beta})$ equal to $([2,0],[0,0])$ or $([0,0],[2,0])$ satisfy {\bf R1}--{\bf R6}.  Thus $a_2b_0,b_2b_0 \in \mathcal{R}_2$.  For index $[2]$, {\bf R5} implies $\alpha_1 \neq 2$.  Thus, $a_2\notin\mathcal{R}_2$.  The monomial with $(\bs{\alpha},\bs{\beta})=([0],[2])$ satisfies {\bf R1}--{\bf R6}.  Thus $b_2 \in \mathcal{R}_2$.  Finally,
 $\mathcal{R}_2=\{b_2b_1b_0,b_2a_1,a_2b_0,b_2b_0,b_2\}$.
\label{ex:TkExample}
\end{example}

The following remark gives conditions for when monomials $a_k$ or $b_k$ are in $\mathcal{R}_k$.

\begin{remark}
For $k>0$, the monomial $a_k$ with $(\bs{\alpha},\bs{\beta})=([k],[0])$ is an element of $\mathcal{R}_k$ if and only if $k \equiv 0,1 (\bmod 3)$ by {\bf R5}.  Similarly by {\bf R6}, the monomial $b_k$ with $(\bs{\alpha},\bs{\beta})=([0],[k])$ is an element of $\mathcal{R}_k$ if and only if $k\equiv 0,2 (\bmod 3)$.  Thus for $i\geq 0$, 
$\{a_{3i+1}\}=\mathcal{R}_{3i+1}\cap\{a_{3i+1},b_{3i+1}\}$, $\{b_{3i+2}\}=\mathcal{R}_{3i+2}\cap \{a_{3i+2},b_{3i+2}\}$, and $\{a_{3i+3},b_{3i+3}\}\subset \mathcal{R}_{3i+3}$.
\label{rmk:monomialakbk}
\end{remark}

\begin{lem}  The sequence $r_k-|\rho(k)|$ counts the number of elements in $\mathcal{R}_k$.  
\label{lem:K3SetSize}
\end{lem}
\begin{proof}
Our proof uses induction on $k$.  From Example \ref{ex:TkExample}, the sets $\mathcal{R}_0$, $\mathcal{R}_1$, and $\mathcal{R}_2$ have 1, 2, and 5 elements, respectively.  We verify $r_0-|\rho(0)|=1-0=1$, $r_1-|\rho(1)|=3-1=2$, and $r_2-|\rho(2)|=5-0=5$.

Make the induction hypothesis that $\mathcal{R}_{k-3}$, $\mathcal{R}_{k-2}$, and $\mathcal{R}_{k-1}$ have $r_{k-3}-|\rho(k-3)|$, $r_{k-2}-|\rho(k-2)|$, and $r_{k-1}-|\rho(k-1)|$ elements, respectively.  Let $\overline{\mathcal{R}}_{j}$ be the monomials of $\mathcal{R}_{j}$ after substitutions $a_{j+1}\mapsto \overline{a_{j+1}}$ and $b_j\mapsto\overline{b_j}$ for $j\geq 0$.  Here the overline denotes
a different copy of the indeterminates.

We define a bijection 
$\psi:\overline{\mathcal{R}}_{k-3} \cup\mathcal{R}_{k-3} \cup \overline{\mathcal{R}}_{k-2} \cup  
\mathcal{R}_{k-2} \cup \mathcal{R}_{k-1} \rightarrow \mathcal{R}_k \backslash \{ a_k , b_k \}$ 
as follows.  
$\psi$ left multiplies elements of $\overline{\mathcal{R}}_{k-3} \cup \overline{\mathcal{R}}_{k-2}$ by $\overline{a_k}$ and then removes all overlines, $\psi$ left multiplies elements of $\mathcal{R}_{k-2}\cup\mathcal{R}_{k-1}$ by $b_k$, and $\psi$ leaves each element of $\mathcal{R}_{k-3}$ fixed.  $\psi^{-1}$ is described as follows.  When $a_k$ is a factor of a monomial in $\mathcal{R}_k \backslash \{ a_k , b_k \}$, $\psi^{-1}$ removes the factor $a_k$ and overlines the remaining indeterminate factors.  The result of this is in $\overline{\mathcal{R}}_{k-3}$ or $\overline{\mathcal{R}}_{k-2}$ due to property {\bf R3}.  Similarly, when $b_k$ is a factor of a monomial in $\mathcal{R}_k \backslash \{ a_k , b_k \}$, $\psi^{-1}$ removes the factor $b_k$, and the result is in either $\mathcal{R}_{k-2}$ or $\mathcal{R}_{k-1}$ by property {\bf R4}.  Otherwise, $\psi^{-1}$ leaves a monomial of $\mathcal{R}_k \backslash \{ a_k , b_k \}$ fixed.

Since there is a bijection between $\mathcal{R}_k \backslash \{ a_k , b_k \}$ and the pairwise disjoint union $\overline{\mathcal{R}}_{k-3} \cup\mathcal{R}_{k-3} \cup \overline{\mathcal{R}}_{k-2} \cup  \mathcal{R}_{k-2} \cup \mathcal{R}_{k-1}$, the number of elements in $\mathcal{R}_k \backslash \{ a_k , b_k \}$ is 
\[r_{k-1}-|\rho(k-1)|+2r_{k-2}-2|\rho(k-2)|+2r_{k-3}-2|\rho(k-3)|.\]
By the recurrence formula for $r_k$ in Corollary \ref{cor:TNumTerms}, the above equals
\begin{equation}
r_k-|\rho(k-1)|-2|\rho(k-2)|-2|\rho(k-3)|.
\label{eqn:CountR_kWOakbk}
\end{equation}

Remark \ref{rmk:monomialakbk} gives that the number of elements in $\mathcal{R}_k \cap \{a_k,b_k \}$ is one when $k \equiv 1,2 (\bmod 3)$ and two when $k \equiv 0 (\bmod 3)$.  Since $|\rho(k)|=1$ if $k \equiv 1 (\bmod 3)$ and is zero otherwise,  the total number of elements in $\mathcal{R}_k \cap \{a_k,b_k \}$ is expressible as $|\rho(k)|+|\rho(k-1)|+2|\rho(k-2)|$ or $|\rho(k-3)|+|\rho(k-1)|+2|\rho(k-2)|$.  Adding this to the number of elements of $\mathcal{R}_k\backslash\{ a_k , b_k \}$ found in \eqref{eqn:CountR_kWOakbk} gives that the number of monomials in $\mathcal{R}_k$ is $r_k-|\rho(k-3)|=r_k-|\rho(k)|.$
\qed
\end{proof}

\begin{cor} Let $s_k=|\mathcal{R}_k|$. Then $s_k$ satisfies the linear 
recurrence $s_k=s_{k-1}+2s_{k-2}+3s_{k-3}-s_{k-4}-2s_{k-5}-2s_{k-6}$ with initial conditions,
$s_0=1$, $s_1=2$, $s_2=5$, $s_3=13$, $s_4=28$, and $s_5=65$.
\end{cor}

\begin{proof} From the fact that $s_k$ and $|\rho(k)|$ satisfy linear recurrences of order three,
with constant coefficients, it follows that $s_k$ can satisfy a similar recurrence of order at most
$9$. Standard linear algebra gives the recurrence for $s_k$.
\qed
\end{proof}

\begin{prop}  
For $k\geq 0$, $\mathrm{supp}(R_k-\rho(k))=\mathcal{R}_k$.
\label{prop:TkSubscrpt}
\end{prop}
\begin{proof}
By Lemma \ref{lem:K3SetSize} $\mathrm{supp}(R_k-\rho(k))$ and $\mathcal{R}_k$ have the same number of elements, $r_k-|\rho(k)|$.  Therefore, it is enough to show that the support of $R_k-\rho(k)$ is a subset of $\mathcal{R}_k$.  Our proof of this uses induction on $k$.  The support of $R_0-\rho(0)$ is $\{b_0\}$, the support of $R_1-\rho(1)$ is $\{a_1,b_1b_0\}$, and the support of $R_2-\rho(2)$ is $\{a_2b_0,b_2,b_2a_1,b_2b_1b_0,b_2b_0\}$.  These sets are identical to the corresponding sets $\mathcal{R}_0$, $\mathcal{R}_1$, and $\mathcal{R}_2$ found in Example \ref{ex:FirstTsets}.

Let $m$ be a monomial of $R_k$ with degree $\ell$, $a$-index $\bs{\alpha}$, $b$-index $\bs{\beta}$, and index $\bs{\lambda}$.  Suppose that each monomial $m'$ in the support of $R_i$ is also in $\mathcal{R}_i$ for $i=1,2,\dots,k-1$.  By \eqref{eqn:TRecurrence} and Lemma \ref{lem:TDisjointTerms}, $m$ is either in the support of $-R_{k-3}-\rho(k)$, $a_kR_{k-2}$, $a_kR_{k-3}$, $b_kR_{k-1}$, or $b_kR_{k-2}$.  We verify that the index, $a$-index, and $b$-index of $m$ satisfy the conditions {\bf R1}--{\bf R6} in each of these cases.

Suppose $m$ is a monomial of $-R_{k-3}-\rho(k)$.  Then $m$ is a monomial of $R_{k-3}-\rho(k-3)$, since $-\rho(k)=\rho(k-3)$ and the supports of polynomials $-P$ and $P$ are the same.  By the induction hypothesis, $m$ is in $\mathcal{R}_{k-3}$.  From property {\bf R2} for $\mathcal{R}_{k-3}$, the first component of the index of $m$ satisfies $\lambda_1\equiv k-3 (\bmod 3)$, so $\lambda_1\equiv k (\bmod 3)$ and $m$ satisfies {\bf R2}.  The other properties {\bf R1}, {\bf R3}-{\bf R6} clearly follow from the respective properties of $\mathcal{R}_{k-3}$.

Suppose $m$ is a monomial of $a_kR_{k-3+p}$ where $p=0,1$.  If $\ell=1$, then $m=a_k$ and $\rho(k-3+p)$ is nonzero.  Thus, $k-3+p \equiv 1 (\bmod 3)$ and $k \equiv 1,0 (\bmod 3)$, and {\bf R5} holds.  The monomial $a_k$ has index $\bs{\lambda}=[k]+[0]$.  Properties {\bf R1} and {\bf R2} of $\mathcal{R}_k$ hold for $a_k$.  Properties {\bf R4} and {\bf R6} hold for $a_k$ since no $\lambda_j=\beta_j$.  Property {\bf R3} holds since $\ell=1$.  Next, if $\ell>1$, then by the inductive hypothesis $m=a_km'$, where $m'\in \mathcal{R}_{k-3+p}$.  The index of $m$ is $\bs{\lambda}=[k,\bs{\alpha'}]+[0,\bs{\beta'}]$, where $\bs{\lambda}'=\bs{\alpha}'+\bs{\beta}'$ is the index of $m'$.  Clearly $m$ satisfies {\bf R1} and {\bf R2} for $\mathcal{R}_k$.  From {\bf R2} for $\mathcal{R}_{k-3+p}$, $\lambda_1'\equiv k-3+p (\bmod 3)$, thus $\lambda_1=\alpha_1=k\not \equiv k-3+p+1 (\bmod 3)$ and {\bf R3} holds for $j=1$.  Property {\bf R3} holds for $j>1$ since $m'\in \mathcal{R}_{k-3+p}$.  Properties {\bf R4}--{\bf R6} are satisfied by $m$ from the respective properties of $m'\in \mathcal{R}_{k-3+p}$.  

Suppose $m$ is a monomial of $b_kR_{k-2+p}$ where $p=0,1$.  If $\ell=1$, then $m=b_k$ and $\rho(k-2+p)$ is nonzero.  Thus, $k-2+p \equiv 1 (\bmod 3)$ and $k \equiv 0,2 (\bmod 3)$, and {\bf R6} holds.  The monomial $b_k$ has index $\bs{\lambda}=[0]+[k]$.  Properties {\bf R1} and {\bf R2} of $\mathcal{R}_k$ hold for $b_k$.  Properties {\bf R3} and {\bf R5} of $\mathcal{R}_k$ hold for $b_k$ since no $\lambda_j=\alpha_j$.  Property {\bf R4} holds since $\ell=1$.  Next if $\ell>1$, then by the inductive hypothesis $m=b_km'$, where $m' \in \mathcal{R}_{k-2+p}$.  The index of $m$ is $\bs{\lambda}=[0,\bs{\alpha}']+[k,\bs{\beta}']$, where $\bs{\lambda}'=\bs{\alpha}'+\bs{\beta}'$ is the index of $m'$.  Clearly $m$ satisfies {\bf R1} and {\bf R2} for $\mathcal{R}_k$.  From {\bf R2} for $\mathcal{R}_{k-2+p}$, $\lambda' \equiv k-2+p (\bmod 3)$, thus $\lambda_1=\beta_1=k \not \equiv k-2+p (\bmod 3)$ and {\bf R4} holds for $j=1$.  Property {\bf R4} holds for $j>1$ since $m'\in \mathcal{R}_{k-2+p}$.  Properties {\bf R3}, {\bf R5}, and {\bf R6} are satisfied by $m$ from the respective properties of $m' \in \mathcal{R}_{k-2+p}$. 
\qed
\end{proof}

We now turn our attention to the coefficients of $R_k$.  By Lemma \ref{lem:CfracCoefLemma}, each monomial $m\in \mathrm{supp}(R_k)$ has coefficient $c_m=\pm 1$.  The determination of the sign depends on the following definition.

\begin{definition}
We call a set of three consecutive integers an \emph{adjacent triple}.  For a monomial $m$ of $\mathcal{R}_k$ with index $\bs{\lambda}$, the integers in the set 
\[
\{-1,0,\dots,k\} \backslash \{\lambda_1,\lambda_2,\dots,\lambda_\ell\}
\]
 are called the \emph{omitted subscripts of $m$}.  For a monomial $m\in\mathcal{R}_k$ with index $\bs{\lambda}$, define the function $g_k(m)$ to be the maximum number of pairwise disjoint adjacent triples whose union is a \emph{subset} of the omitted subscripts of $m$.
\label{defn:AdjTrip}
\end{definition}

The coefficient $c_m$ of $m \in \mathrm{supp}(R_k)$ is determined by the parity of $g_k(m)$. Specifically, $c_m=(-1)^{g_k(m)}$.  We show this in Lemma \ref{lem:TermSign}. The coefficients of three monomials are computed in Example \ref{ex:TermSigns}.  

\begin{example}  First, consider the monomial $b_2$ in $\mathrm{supp}(R_5)$ with index $[2]$.  Monomial $b_2$ has omitted subscripts $\{5,4,3,1,0,-1\}$.  This set is the union of $g_5(b_2)=2$ disjoint triples: $\{5,4,3\}$ and $\{1,0,-1\}$.  Thus, the coefficient of $b_2$ is $(-1)^{2}=1$.

Second, consider the monomial $b_6b_4a_3$ in $\mathrm{supp}(R_6)$ with index $[6,4,3]$.  It has omitted subscripts $\{5,2,1,0,-1\}$.  The omitted subscripts contain two adjacent triples, $\{2,1,0\}$ and $\{1,0,-1\}$.  Since these adjacent triples are not disjoint, $g_6(b_6b_4a_3)=1$, and the coefficient of $b_6b_4a_3$ is $(-1)^1=-1$.

Third, consider the monomial $a_6b_4b_2b_1b_0$ in $\mathrm{supp}(R_6)$. This monomial has 
index $[6,4,2,1,0]$ and has omitted subscripts $\{5,3,-1\}$.  The omitted subscripts 
give $g_6(a_6b_4b_2b_1b_0)=0$ pairwise disjoint adjacent triple subsets.  Thus the coefficient of $a_6b_4b_2b_1b_0$ is $(-1)^0=1$.
\label{ex:TermSigns}
\end{example}

Observe that the coefficient of $b_2$ in $R_8$ should be the opposite of its coefficient in $R_5$, since the omitted subscripts in the former case contains an additional adjacent triple $\{8,7,6\}$.  More generally, Lemma \ref{lem:SignRec} describes how the coefficient of a monomial of $R_k$ is based upon recurrence formula \eqref{eqn:TRecurrence}.

\begin{lem}
For each monomial $m$ of $R_k-\rho(k)$ with index $\bs{\lambda}=\bs{\alpha}+\bs{\beta} \in \mathcal{R}_k$ of length $\ell>1$, let $m'$ be the monomial with index $\bs{\lambda}'=\bs{\alpha}'+\bs{\beta}'=[\alpha_2,\alpha_3,\dots,\alpha_\ell]+[\beta_2,\beta_3,\dots,\beta_\ell]$.  Define $\sgn_k(m)$ to be the coefficient of $m$ (or sign of $m$) in the polynomial $R_k$.  Then for $k \geq 3$ and $p=1,2,3$,

\[\sgn_k(m)=\begin{cases}
\sgn_{k-p}(m') &  \text{if } \lambda_1=k \text{ and } \lambda_2\equiv k-p (\bmod 3),\\
-\sgn_{k-3}(m) & \text{if }\lambda_1\neq k.
\end{cases}\]
\label{lem:SignRec}
\end{lem}

\begin{proof}
Let $\sgn_k(m)m$ be a term of $R_k-\rho(k)$.   
There are five cases corresponding to the five summands when the right hand side of \eqref{eqn:TRecurrence} is expanded.  
If $\sgn_k(m)m$ is a term of $b_kR_{k-1}$, then $\sgn_k(m)=\sgn_{k-1}(m')$.
If $\sgn_k(m)m$ is a term of $a_kR_{k-2}$ or $b_kR_{k-2}$, then $\sgn_k(m)=\sgn_{k-2}(m')$.
If $\sgn_k(m)m$ is a term of $a_kR_{k-3}$, then $\sgn_k(m)=\sgn_{k-3}(m')$.
If $\sgn_k(m)m$ is a term of $-R_{k-3}$, then $\sgn_k(m)=-\sgn_{k-3}(m)$.
\qed
\end{proof}

\begin{lem}
\[g_k(m)=\left\lfloor \dfrac{k-\lambda_1}{3}\right\rfloor+\left\lfloor \dfrac{\lambda_\ell+1}{3}\right\rfloor+\sum_{j=2}^\ell \left\lfloor \dfrac{\lambda_{j-1}-\lambda_j-1}{3} \right\rfloor,\]
the last sum being zero when $\ell=1$.
\label{lem:gFormula}
\end{lem}
\begin{proof}
The maximum number of disjoint three adjacent triples strictly between integers $\lambda_{j-1}$ and $\lambda_j$ is
\[\left\lfloor \dfrac{\lambda_{j-1}-\lambda_j-1}{3} \right\rfloor.\]
The lemma follows from summing and using the conventions $\lambda_0=k+1$ and $\lambda_{\ell+1}=-2$. 
\qed
\end{proof}

\begin{lem}
For a monomial $m$ of $R_k$,
\[\sgn_k(m)=(-1)^{g_k(m)}.\]
\label{lem:TermSign}
\end{lem}
\begin{proof}
We proceed by induction on $k$.  The initial cases are given by \eqref{eqn:ManyRs}.  Let $m$ be in the support of $R_k-\rho(k)$ with index $\bs{\lambda}=\bs{\alpha}+\bs{\beta}$.  Suppose that each monomial $m^*$ in the support of $R_j-\rho(j)$ has coefficient $\sgn_j(m^*)=(-1)^{g_j(m^*)}$, for $j=0,1,2,\dots,k-1$.  From Lemma \ref{lem:SignRec} and the induction hypothesis,
\begin{equation}
\sgn_k(m)=\begin{cases}
(-1)^{g_{k-1}(m')} &  \text{if } \lambda_1=k \text{ and } \lambda_2\equiv k-1 (\bmod 3),\\
(-1)^{g_{k-2}(m')} &  \text{if } \lambda_1=k \text{ and } \lambda_2\equiv k-2 (\bmod 3),\\
(-1)^{g_{k-3}(m')} &  \text{if } \lambda_1=k \text{ and } \lambda_2\equiv k-3 (\bmod 3),\\
(-1)^{1+g_{k-3}(m)} & \text{if }\lambda_1\neq k,
\end{cases}
\label{eqn:TermSignHelp}
\end{equation}
where $m'$ has index $\bs{\lambda}'=[\alpha_2,\alpha_3,\dots,\alpha_\ell]+[\beta_2,\dots,\beta_\ell]$.  Each of the cases in \eqref{eqn:TermSignHelp} gives $(-1)^{g_k(m)}$.  Indeed, in each case where $\lambda_1=k$, 
\[
\begin{split}
\{-1,0,1,2,3,\dots,k\} &\backslash \{\lambda_1,\lambda_2,\dots,\lambda_\ell\}\\
&=\{-1,0,1,2,3,\dots,k-1\} \backslash \{\lambda_2,\lambda_3\dots,\lambda_\ell\}.
\end{split}
\]
So when $\lambda_1=k$, the maximum number of disjoint adjacent triples of these equal sets $g_k(m)$ and $g_{k-1}(m')$, respectively, are equal.  

We use that $g_k(m)=g_{k-1}(m')$ for the first three cases where $\lambda_1=k$ and $\lambda_2\equiv k-p (\bmod 3)$ for $p=1,2,3$.  Let $d$ be the integer such that $\lambda_2=k-p-3d$.  By Lemma \ref{lem:gFormula},
\begin{equation}
g_{k-1}(m')=\left\lfloor \dfrac{\lambda_\ell+1}{3}\right\rfloor
+\left\lfloor \dfrac{k-1-\lambda_2}{3}\right\rfloor
+\sum_{j=3}^\ell \left\lfloor \dfrac{\lambda_{j-1}-\lambda_j-1}{3} \right\rfloor.
\label{eqn:TermSignHelp2}
\end{equation}
Substituting $k-p-3d$ for $\lambda_2$ in the second summand of \eqref{eqn:TermSignHelp2} yields,
\[
\left\lfloor \dfrac{k-1-\lambda_2}{3}\right\rfloor=\left\lfloor \dfrac{3d+p-1}{3}\right\rfloor=d=\left\lfloor \dfrac{3d}{3} \right\rfloor=\left\lfloor \dfrac{k-p-\lambda_2}{3}\right\rfloor.
\]
Thus we can replace the second summand of \eqref{eqn:TermSignHelp2}:
\[g_{k-1}(m')=\left\lfloor \dfrac{\lambda_\ell+1}{3}\right\rfloor
+\left\lfloor \dfrac{k-p-\lambda_2}{3}\right\rfloor
+\sum_{j=3}^\ell \left\lfloor \dfrac{\lambda_{j-1}-\lambda_j-1}{3} \right\rfloor=g_{k-p}(m').\]

For the last case when $\lambda_1 \neq k$, property {\bf R2}  gives that $\lambda_1\neq k, k-1, k-2$.  Thus, the adjacent triple $k,k-1,k-2$ is in 
\[
\{-1,0,1,2,3,\dots,k\} \backslash \{\lambda_1,\lambda_2,\lambda_3\dots,\lambda_\ell\}, 
\] 
and $g_{k}(m)$ is one more than the number of adjacent triples in 
\[
\{-1,0,1,2,3,\dots,k-3\} \backslash \{\lambda_1,\lambda_2,\lambda_3\dots,\lambda_\ell\}. 
\]
Thus $1+g_{k-3}(m)=g_k(m)$ and  $\sgn_k(m)=(-1)^{g_k(m)}$ by \eqref{eqn:TermSignHelp}.
\qed
\end{proof}

\section{Main Results} \label{main}

The following proposition gives a combinatorial description for the terms of the polynomials $R_k$. 

\begin{prop}  For $k\geq 0$,
\begin{equation}R_k(a_1,a_2,\dots,a_k,b_0,b_1,\dots,b_k)=\rho(k)+\sum_{m \in \mathcal{R}_k} (-1)^{g_k(m)} m.
\label{eqn:TkEqn}
\end{equation}
\label{thm:R}
\end{prop}
\begin{proof}
Applying Lemmas \ref{lem:CfracCoefLemma}, \ref{prop:TkSubscrpt}, and \ref{lem:TermSign} gives \eqref{eqn:TkEqn}.  
\qed
\end{proof}

Theorem \ref{thm:EMAnalogue} below results from piecing together Lemma \ref{lem:Num2Den}, \eqref{eqn:TPRelation},
and Proposition \ref{thm:R}.
 Lemma \ref{lem:Num2Den} states $Q_k=\phi(-P_k)$, where $\phi$ is the substitution that  maps $a_1$, $b_0$, and $b_1$ to $0$, and for $j>1$, substitutes $a_{j-1}$ for $a_j$, and substitutes
 $b_{j-1}$ for $b_j$. The linearity of $\phi$ and \eqref{eqn:TPRelation} gives $Q_k=-\phi(R_{k+1})-\phi(R_k)$.

Define $\sigma(k)$ to be
\begin{equation}
\sigma(k)=\dfrac{2 \sqrt{3}}{3} \sin \left( \dfrac{k\pi}{3} \right) = \dfrac{2 \sqrt{3}}{3} \mathrm{Im} \left( e^{k\pi i/3} \right),
\label{eqn:BetaDefn}
\end{equation}
and note that $\sigma(k)$ is the six-periodic sequence which begins $0$, $1$, $1$, $0$, $-1$, 
$-1$, $\dots$, for $k\geq 0$ that satisfies $\sigma(k)=-\sigma(k-3)$.  

\begin{thm}
The $k$th classical numerator and denominator $P_k$ and $Q_k$ of
\[ b_0+\kfrac{-1+a_1}{1+b_1}\kfrac{-1+a_2}{1+b_2}\kfrac{-1+a_3}{1+b_3}\lowfrac{\cdots} \]
are
\begin{equation}
P_k=-\sigma(k)+\sum_{m \in \mathcal{R}_{k-1}} (-1)^{g_{k-1}(m)} m+\sum_{m \in \mathcal{R}_k} (-1)^{g_k(m)} m,
\label{eqn:ClassNum}
\end{equation}
and
\begin{equation}
Q_k=\sigma(k+1)-\sum_{\substack{m \in \mathcal{R}_k \\ \lambda_\ell>1}} (-1)^{g_{k}(m)} \phi(m)-\sum_{\substack{m \in \mathcal{R}_{k+1} \\ \lambda_\ell>1}} (-1)^{g_{k+1}(m)} \phi(m).
\label{eqn:ClassDen}
\end{equation}
\label{thm:EMAnalogue}
\end{thm}

\begin{proof}
First observe that $\sigma(k)=-\rho(k-1)-\rho(k)$.  The proof follows immediately from \eqref{eqn:TPRelation}, \eqref{eqn:TkEqn}, and \eqref{eqn:Num2Den}.  The condition $\lambda_\ell>1$ in \eqref{eqn:ClassDen} simplifies the summations by removing all summands with $\phi(m)=0$, where $\phi(m)$ is as defined after the proof of Proposition \ref{thm:R}.  
\qed
\end{proof}

The next two corollaries are specializations of Theorem \ref{thm:EMAnalogue}.  The first of these is the case when $b_j=0$ for $j\geq 0$.  Making this substitution into \eqref{eqn:ClassNum} causes all monomials whose $b$ index has nonzero components and those elements with $\beta_\ell=\alpha_\ell=0$ to vanish from the summation in \eqref{eqn:ClassNum}.  This implies that the index equals the $a$ index for these monomials.  Thus for $k\geq 1$ the support of the classical numerators is the subset $\mathcal{U}_k$ of $\mathcal{R}_k$  whose $b$ index is zero and whose index satisfies the following properties:
\begin{description}
	\item[U1] $k \geq \lambda_1 > \lambda_2 > \dots \lambda_\ell \geq 1$.
	\item[U2] $\lambda_1\equiv k (\bmod 3)$.
	\item[U3] $\lambda_j \not \equiv \lambda_{j+1}+1 (\bmod 3)$.
	\item[U4] $\lambda_\ell \not \equiv 2 (\bmod 3)$.
\end{description}

\begin{cor}
The $k$th classical numerator and denominator $C_k$ and $D_k$ of
\[\kfrac{-1+a_1}{1}\kfrac{-1+a_2}{1}\kfrac{-1+a_3}{1}\lowfrac{\cdots} \]
are
\begin{equation}
C_k=-\sigma(k)+\sum_{m \in \mathcal{U}_{k-1} } (-1)^{g_{k-1}(m)} m+\sum_{m \in \mathcal{U}_k } (-1)^{g_k(m)} m,
\label{eqn:NumSeqClassNum}
\end{equation}
and
\begin{equation}
D_k=\sigma(k+1)-\sum_{\substack{m \in \mathcal{U}_k \\ \lambda_\ell>1}} (-1)^{g_{k}(m)}\phi(m)-\sum_{\substack{m\in \mathcal{U}_{k+1} \\ \lambda_\ell>1}} (-1)^{g_{k+1}(m)} \phi(m).
\label{eqn:NumSeqClassDen}
\end{equation}
\label{cor:NumSeqEMAnalogue}
\end{cor}

The second case of Theorem \ref{thm:EMAnalogue} is $b_0=0$ and $a_j=0$ for $j\geq 1$.  The monomials whose $a$ index has nonzero components as well as those elements with $\beta_\ell=\alpha_\ell=0$ vanish from the summation in \eqref{eqn:ClassNum}. Hence for $k\geq 1$ the support of the classical numerators is the subset $\mathcal{V}_k$ of $\mathcal{R}_k$ whose $a$ index is zero and whose index satisfies the following properties:
\begin{description}
	\item[V1] $k \geq \lambda_1 > \lambda_2 > \dots \lambda_\ell \geq 2$.
	\item[V2] $\lambda_1\equiv k (\bmod 3)$.
	\item[V3] $\lambda_j \not \equiv \lambda_{j+1} (\bmod 3)$.
	\item[V4] $\lambda_\ell \not \equiv 1 (\bmod 3)$.
\end{description}

\begin{cor}
The $k$th classical numerator and denominator $G_k$ and $H_k$ of
\[\kfrac{-1}{1+b_1}\kfrac{-1}{1+b_2}\kfrac{-1}{1+b_3}\lowfrac{\cdots} \]
are
\begin{equation}
G_k=-\sigma(k)+\sum_{m \in \mathcal{V}_{k-1}} (-1)^{g_{k-1}(m)} m+\sum_{m \in \mathcal{V}_k} (-1)^{g_k(m)} m,
\label{eqn:DenSeqClassNum}
\end{equation}
and
\begin{equation}
H_k=\sigma(k+1)-\sum_{\substack{m\in \mathcal{V}_k \\ \lambda_\ell>1}} (-1)^{g_{k}(m)}\phi(m)-\sum_{\substack{m\in \mathcal{V}_{k+1}\\ \lambda_\ell>1}} (-1)^{g_{k+1}(m)} \phi(m).
\label{eqn:DenSeqClassDen}
\end{equation}
\label{cor:DenSeqEMAnalogue}
\end{cor}
Corollary \ref{cor:DenSeqEMAnalogue} was first given as Theorem 33 of \cite{Schaumburg2015}.

We refer to a finite decreasing sequence $\lambda_i$ satisfying 
$\lambda_j \not \equiv \lambda_{j+1}+t (\bmod 3)$, for a fixed $t\in\{0,1,2\}$, as an \emph{alternating triality sequence}. Only the 
cases $t=0,1$ occur in this paper.

\section{Applications to Polynomial Identities and Integer Sequences} \label{sec:last}

\subsection{Relating Theorems \ref{thm:NoncomEM} and \ref{thm:EMAnalogue}} \label{subsec:relating}
What do Theorems \ref{thm:NoncomEM} with \ref{thm:EMAnalogue} imply when taken together? 

To relate these theorems, the following change of variables is used.
Let $\delta:\mathbb{Z}[\mathcal{M}]\to\mathbb{Z}[\mathcal{M}]$ be the homomorphism induced
by $\delta(a_i)=-1+a_i$ and $\delta(b_i)=1+b_i$. When applied to $\mathcal{A}_k$, each monomial
of length $l$ gives rise to $2^l$ monomials with different signs, since the monomials in 
$\mathcal{A}_k$ are of degree one in each of their indeterminates.  Thus Theorems \ref{thm:NoncomEM} and \ref{thm:EMAnalogue} give,
\begin{equation}
\sum_{m \in \mathcal{A}_k} \delta(m)=
-\sigma(k)+\sum_{m \in \mathcal{R}_{k-1}} (-1)^{g_{k-1}(m)} m+\sum_{m \in \mathcal{R}_k} (-1)^{g_k(m)} m.
\label{eqn:NoncomEMNumK}
\end{equation}
Notice the left side of \eqref{eqn:NoncomEMNumK} has intense cancellation, while the right side has none. This identity becomes more explicit in the special cases corresponding to 
Corollaries \ref{cor:NumSeqEMAnalogue} and \ref{cor:DenSeqEMAnalogue}.
First we provide a corollary of Theorem \ref{thm:NoncomEM} that can be equated
to Corollary \ref{cor:NumSeqEMAnalogue}. 
Define the set of minimal difference two sequences $\bs{C}_k$ by
\begin{equation*}
\bs{C}_k=\{\bs\lambda:k \geq \lambda_1 >^2 \lambda_2 >^2 \dots >^2 \lambda_\ell = 1\}.
\end{equation*}
It is easy to see that $|\bs{C}_k|$ equals the $k$th Fibonacci number, for an element of $\bs{C}_k$ is either an element
of $\bs{C}_{k-1}$, or is obtained by adjoining the integer $k$ to an element of $\bs{C}_{k-2}$.
The initial values $F_0=|\bs{C}_0| = 0$ and $F_1=|\bs{C}_1| = 1$ give the conclusion.

\begin{cor}
The classical numerators of the continued fraction 
\begin{equation}
\kfrac{-1+a_1}{1}\kfrac{-1+a_2}{1}\kfrac{-1+a_3}{1}\lowfrac{\cdots}
\label{eqn:K3NumCF}
\end{equation}
in noncommutative indeterminates $\{a_{j+1}\}_{j \geq 0}$ for $k\geq 0$ are given by
\begin{equation}
C_k=\sum_{ \bs{\lambda} \in \bs{C}_k} (-1+a_{\lambda_1})(-1+a_{\lambda_2})\cdots(-1+a_{\lambda_\ell}).
\label{eqn:NumSeqClassNumEM}
\end{equation}
\label{NumSeqEM}
\end{cor}
\begin{proof}
By Theorem \ref{thm:EM} (with $b_0=0$ and $b_i=1$, for $i>0$), the $k$th classical numerator of
\[\kfrac{a_1}{1}\kfrac{a_2}{1}\kfrac{a_3}{1}\lowfrac{\cdots}\]
equals
\[\sum_{\bs{\lambda} \in \bs{C}_k} a_{\lambda_1}a_{\lambda_2}\cdots a_{\lambda_\ell}\]
after writing subscripts in descending order.  Substituting sequence $\{-1+a_j\}_{j\geq 1}$ for $\{a_j\}_{j\geq 1}$ yields  \eqref{eqn:NumSeqClassNumEM}. 
\qed
\end{proof}

Equating Corollaries  \ref{cor:NumSeqEMAnalogue} and \ref{NumSeqEM} gives Corollary \ref{cor:NumSeqEqual} below.

\begin{cor}\label{cor:ck}
\begin{equation}
\sum_{\bs{\lambda} \in \bs{C}_k} \prod_{j=1}^\ell(-1+a_{\lambda_j})=-\sigma(k)+\sum_{m \in \mathcal{U}_{k-1} } (-1)^{g_{k-1}(m)} m+\sum_{m \in \mathcal{U}_k } (-1)^{g_k(m)} m,
\label{eqn:NumSeqEqualGen}
\end{equation}
\begin{equation}
\sum_{\bs{\lambda} \in \bs{C}_k} (-1)^\ell=-\sigma(k),
\label{eqn:NumSeqEqualSpec1}
\end{equation}
and
\begin{equation}
-\sigma(k)+\sum_{m \in \mathcal{U}_{k-1} } (-1)^{g_{k-1}(m)}+\sum_{m \in \mathcal{U}_k } (-1)^{g_k(m)} =0
\label{eqn:NumSeqEqualSpec2}
\end{equation}

\label{cor:NumSeqEqual}
\end{cor}
Example \ref{ex:CkUk} below demonstrates these identities in the $k=5$ case. Before the example we give the simple proof of 
\eqref{eqn:NumSeqEqualSpec1} mentioned in Section \ref{overview:Results}.

\begin{proof}
Let $e_n$ and $o_n$ denote the number of elements of $\mathbf{C}_n$ of even and odd lengths, respectively. Since every element of $\mathbf{C}_n$ is either an element of $\mathbf{C}_{n-1}$,
or is obtained by adjoining the integer $n$ to an element of $\mathbf{C}_{n-2}$, it is clear that
$o_n = o_{n-1} + e_{n-2}$, and $e_n = e_{n-1} + o_{n-2}$. Putting $x_n = e_n - o_n$, and subtracting
the first of the two equations from the second, gives that $x_n=x_{n-1}-x_{n-2}$. Clearly
$x_{n-1}=x_{n-2}-x_{n-3}$. Substituting the second of these two equations into the first gives
$x_n=-x_{n-3}$, which is the same recurrence satisfied by $-\sigma(k)$. For $k=0,1,2$, $-\sigma(k)=0,-1,-1$, and the left-hand side of \eqref{eqn:NumSeqEqualSpec1} also equals $0,-1,-1$, since $C_0 = \emptyset$ and $C_1$ and $C_2$ each contain only the sequence $\{ 1\}$.
\qed
\end{proof}

\begin{example}
\label{ex:CkUk}
Let $k=5$ in \eqref{eqn:NumSeqEqualGen}.  
It is found that $\sigma(5)=-1$, $\mathcal{U}_4=\left\{[4,1],[4],[1]\right\}$ and 
$\mathcal{U}_5=\left\{[5,3,1],[5,3]\right\}$,
so that right-hand side of \eqref{eqn:NumSeqEqualGen} is:
\begin{equation}
1+a_4a_1-a_4-a_1+a_5a_3a_1-a_5a_3.
\label{eqn:CkUk}
\end{equation}
The left-hand side of \eqref{eqn:NumSeqEqualGen} for $k=5$ can be computed by summing the contributions from each sequence in $\bs{C}_k$ and making several cancellations.
First, we find the contribution due to $\{5,3,1\}\in\mathbf{C}_5$: 
\begin{equation}
(a_5-1)(a_3-1)(a_1-1)=a_5a_3a_1-\left( a_5a_3 + a_5a_1 + a_3a_1 \right)+\left(a_5+a_3+a_1\right)-1.
\label{eqn:CkUk1}
\end{equation}
Similarly, the sequence $\{5,1\}$ contributes
\begin{equation}
(a_5-1)(a_1-1)=a_5a_1-a_5-a_1+1.
\label{eqn:CkUk2}
\end{equation}
The sequence $\{4,1\}$ contributes
\begin{equation}
(a_4-1)(a_1-1)=a_4a_1-a_4-a_1+1.
\label{eqn:CkUk3}
\end{equation}
The sequence $\{3,1\}$ contributes
\begin{equation}
a_3a_1-a_3-a_1+1.
\label{eqn:CkUk4}
\end{equation}
Finally, the sequence $\{1\}$ contributes
\begin{equation}
a_1-1.
\label{eqn:CkUk5}
\end{equation}
Summing \eqref{eqn:CkUk1}--\eqref{eqn:CkUk5} and canceling terms recovers \eqref{eqn:CkUk}.  

The sum on the left-hand side of \eqref{eqn:NumSeqEqualSpec1} in the $k=5$ case is over the sequences
 $\{5,3,1\}$, $\{5,1\}$, $\{4,1\}$, $\{3,1\}$ and $\{1\}$.  The sum can be computed according to the lengths of these sequences as $(-1)^3+3(-1)^2+(-1)^1=1$, which indeed is equal to $-\sigma(5)$. 
Also observe that the substitution $a_j \mapsto 1$ in \eqref{eqn:CkUk} yields zero, 
giving \eqref{eqn:NumSeqEqualSpec2}.  
\end{example}

To relate Theorem \ref{thm:NoncomEM} to Corollary \ref{cor:DenSeqEMAnalogue}, define $\bs{D}_k$ to be the set of alternating parity sequences  satisfying:
\begin{description}
\item[D1] $k \geq \lambda_1 > \lambda_2 > \dots > \lambda_\ell \geq 2$.
\item[D2] $\lambda_1 \equiv k (\bmod 2)$.
\item[D3] $\lambda_j \not \equiv \lambda_{j-1} (\bmod 2)$.
\item[D4] $\lambda_\ell\equiv 0(\bmod 2)$.
\end{description}
Theorem \ref{thm:NoncomEM} now reduces to:
\begin{cor}
The classical numerators of the continued fraction 
\[\kfrac{-1}{1+b_1}\kfrac{-1}{1+b_2}\kfrac{-1}{1+b_3}\lowfrac{\cdots}\]
in noncommutative indeterminates $\{b_j\}_{j \geq 0}$ for $k\geq 0$ are given by
\begin{equation}
G_k= -\chi_1(k)+\sum_{\bs{\lambda} \in \bs{D}_k} (-1)^{\frac{k-\ell+1}{2}}(1+b_{\lambda_1})(1+b_{\lambda_2})\cdots(1+b_{\lambda_\ell}).
\label{eqn:NumSeqClassDenEM}
\end{equation}
\label{DenSeqEM}
\end{cor}
\begin{proof}
By Theorem \ref{thm:EM}, the $k$th classical numerator of 
\[\kfrac{-1}{b_1}\kfrac{-1}{b_2}\kfrac{-1}{b_3}\lowfrac{\cdots}\]
in commuting variables is
\[ b_2 \cdots b_k \left[ 1 + \sum_{3\leq h_1 <^2 h_2 <^2 \cdots <^2 h_j \leq k} 
\dfrac{(-1)^{j+1}}{b_{h_1-1}b_{h_1} b_{h_2-1} b_{h_2} \cdots b_{h_j-1} b_{h_j}} \right].\]
When expanded, the degree of a summand is $\ell=k-1-2j$.  So $j=\frac{k-\ell-1}{2}$.  Note that after cancellation the largest subscript has the same parity as $k$ and the smallest subscript is even.  When $k$ is even, $b_2\cdots b_k$ is not canceled when distributed and the constant is zero.  When $k$ is odd, the constant is $(-1)^{1+(k-1)/2}$.  Thus the constant is $-\chi_1(k)$, the nonprincipal Dirichlet character modulo 4.
Rearranging subscripts in descending order gives that in noncommuting variables, the classical numerator is
\[-\chi_1(k)+\sum_{\bs{\lambda} \in \bs{D}_k} (-1)^{\frac{k-\ell+1}{2}}(1+b_{\lambda_1})(1+b_{\lambda_2})\cdots(1+b_{\lambda_\ell}).\]
\qed
\end{proof}

Equating Corollaries  \ref{cor:DenSeqEMAnalogue} and \ref{DenSeqEM} yields the following
corollary.

\begin{cor}\label{last}
\begin{multline*}
-\chi_1(k)+\sum_{\bs{\lambda} \in \bs{D}_k} (-1)^{\frac{k-\ell+1}{2}}\prod_{j=1}^\ell(1+b_{\lambda_j})
=-\sigma(k)+\sum_{m \in \mathcal{V}_{k-1}} (-1)^{g_{k-1}(m)} m\\+
\sum_{m \in \mathcal{V}_k} (-1)^{g_k(m)} m,\end{multline*}
\[-\chi_1(k)+\sum_{\bs{\lambda} \in \bs{D}_k} (-1)^{\frac{k-\ell+1}{2}}
=-\sigma(k),\]
and
\[
-\chi_1(k)+\sigma(k)=\sum_{m \in \mathcal{V}_{k-1}} (-1)^{g_{k-1}(m)+\ell} +
\sum_{m \in \mathcal{V}_k} (-1)^{g_k(m)+\ell}.
\]
\end{cor}

\subsection{Applications to some linear recurrence sequences}

Corollaries \ref{cor:NumSeqEMAnalogue} and \ref{cor:DenSeqEMAnalogue} lead to new formulas for Fibonacci and Pell numbers.  In this section, we use notation $P_k$ for the $k$th Pell number, not the $k$th classical numerator of $K$ as before. Thus, in this section $P_k=2P_{k-1}+P_{k-2}$, with initial
conditions $P_0=0$ and $P_1=1$.
Despite possible interest, we do not take up the corresponding results that follow from Corollaries \ref{NumSeqEM} and \ref{DenSeqEM} here, nor do we investigate the consequences
for other or more general integer sequences.

The definition of $\sigma(k)$ and the following corollary imply the well-known fact that the $k$th Fibonacci number, $F_k$, is even if and only if $k\equiv 0 (\bmod 3)$.

\begin{cor}
\begin{equation}\label{eqn:Fib1}
F_k=-\sigma(k)+\sum_{\substack{m \in \mathcal{U}_{k-1}  \\ \lambda_\ell>0}} (-1)^{g_{k-1}(m)} 2^\ell+\sum_{\substack{m \in \mathcal{U}_k  \\ \lambda_\ell>0}} (-1)^{g_k(m)} 2^\ell,
\end{equation}
and
\begin{equation}\label{eqn:Fib2}
F_k=\sum_{\substack{m \in \mathcal{U}_{k-1}  \\ \lambda_\ell=1}} (-1)^{g_{k-1}(m)} 2^{\ell-1}+\sum_{\substack{m \in \mathcal{U}_k  \\ \lambda_\ell=1}} (-1)^{g_k(m)} 2^{\ell-1},
\end{equation}
where $F_k$ is the $k$th Fibonacci number.
\label{cor:Fib}
\end{cor}

\begin{proof} The substitution $a_i=1$, $b_0=0$, and $b_j=1$ in \eqref{eqn:CfracNotation} and the recurrence formulas \eqref{eqn:ARecur} and \eqref{eqn:BRecur} gives that $A_k = F_k$
and $B_k=F_{k+1}$.
The same classical numerators and denominators arise from the substitution $a_i=2$ in
Corollary \ref{cor:NumSeqEMAnalogue}. Then \eqref{eqn:NumSeqClassNum} and \eqref{eqn:NumSeqClassDen} gives 
\eqref{eqn:Fib1} along with
\[F_{k+1}=\sigma(k+1)-\sum_{\substack{m \in \mathcal{U}_k \\ \lambda_\ell>1}} (-1)^{g_{k}(m)} 2^\ell -\sum_{\substack{m\in \mathcal{U}_{k+1}  \\ \lambda_\ell>1}} (-1)^{g_{k+1}(m)} 2^\ell.\]
Shifting $k\mapsto k-1$ in this identity and adding it to \eqref{eqn:Fib1} yields \eqref{eqn:Fib2}. 
\qed
\end{proof}

It is also possible to compute $F_k$ using Corollary \ref{cor:DenSeqEMAnalogue}.  The classical numerators of the continued fraction
\begin{equation}
\kfrac{-1}{-1}\kfrac{-1}{1}\kfrac{-1}{-1}\kfrac{-1}{1}\lowfrac{\cdots}
\label{eqn:TauFkCF}
\end{equation}
are $\tau(k)F_k$, where 
\[\tau(k)=\begin{cases}
-1 &\text{if }k\equiv 1,2(\bmod 4)\\
1 & \text{if }k\equiv 3,4(\bmod 4).
\end{cases}\]
The substitutions $b_{2j-1}=-2$ and $b_{2j}=0$ in the continued fraction in Corollary \ref{cor:DenSeqEMAnalogue} give \eqref{eqn:TauFkCF}.
\begin{cor}
\[
\tau(k)F_k=-\sigma(k)+\sum_{m \in \mathcal{V}_{k-1}^\mathrm{odd}} (-1)^{g_{k-1}(m)} (-2)^{\ell}+\sum_{m \in \mathcal{V}_k^\mathrm{odd}} (-1)^{g_k(m)} (-2)^{\ell},
\]
where $\mathcal{V}_{k}^\mathrm{odd}$ is the subset of $\mathcal{V}_{k}$ whose monomials have indices which are all odd.
\end{cor}
Note that when $k$ is even, property {\bf V2} implies $\mathcal{V}_k^\mathrm{odd}=\emptyset$.  Since $\tau(2k-1)=(-1)^k$ and $\tau(2k)=(-1)^k$,
\[(-1)^kF_{2k-1}=-\sigma(2k-1)+\sum_{m \in \mathcal{V}_{2k-1}^\mathrm{odd}} (-1)^{g_{2k-1}(m)} (-2)^{\ell},\]
and
\[(-1)^kF_{2k}=-\sigma(2k)+\sum_{m \in \mathcal{V}_{2k-1}^\mathrm{odd}} (-1)^{g_{2k-1}(m)} (-2)^{\ell}.\]

Turning to the Pell numbers, the $k$th classical numerator of the continued fraction
\begin{equation}
\kfrac{1\slash 4}{1}\kfrac{1\slash 4}{1}\kfrac{1\slash 4}{1}\kfrac{1\slash 4}{1}\lowfrac{\cdots}
\label{eqn:2PowerKPkCF}
\end{equation}
is $P_k \slash 2^{k+1}$.  Substituting $a_i=5/4$ into the continued fraction in Corollary \ref{cor:NumSeqEMAnalogue} yields \eqref{eqn:2PowerKPkCF}.
\begin{cor}
\begin{equation}
P_k=-2^{k+1}\sigma(k)+\sum_{m \in \mathcal{U}_{k-1} } (-1)^{g_{k-1}(m)} \ 5^\ell \ 2^{k+1-2\ell} +\sum_{m \in \mathcal{U}_k } (-1)^{g_k(m)} \ 5^\ell \ 2^{k+1-2\ell},
\label{eqn:Pell1}
\end{equation}
and
\[P_k=2^{k+1}\left(-\sigma(k)+\sum_{m \in \mathcal{U}_{k-1} } (-1)^{g_{k-1}(m)} (5 \slash 4)^\ell+\sum_{m \in \mathcal{U}_k } (-1)^{g_k(m)} (5 \slash 4)^\ell\right).\]
\end{cor}

This gives an interpretation of the fact that Pell number  
\begin{equation}
P_k \equiv \upsilon (k) (\bmod 5),
\label{eqn:PellNumEquiv}
\end{equation}
where $\upsilon(k)$ is the 12 periodic sequence $0,1,2,0,2,4,0,4,3,0,3,1,\dots$ starting from $k=0$.  Observe that multiplying both sides of \eqref{eqn:Pell1} by $2^{k-1}$ gives
\[2^{k-1} P_k \equiv -2^{2k}\sigma(k) \ (\bmod 5),\]
since in the sums $2k-2\ell\geq 0$.  Because $\mathrm{gcd}(2^{k-1},5)=1$,
\[P_k \equiv -2^{k+1}\sigma(k) \ (\bmod 5).\]
The periodicity of $2^k (\bmod 5)$ and $-\sigma(k) (\bmod 5)$ now yield \eqref{eqn:PellNumEquiv}.  

Next, the $k$th classical numerator of the continued fraction
\begin{equation}
\kfrac{-1}{-2}\kfrac{-1}{2}\kfrac{-1}{-2}\kfrac{-1}{2}\lowfrac{\cdots}
\label{eqn:Bell2CF}
\end{equation}
is $\tau(k)P_k$.  We can apply Corollary \ref{cor:DenSeqEMAnalogue} by making substitutions $b_{2k-1}=-3$ and $b_{2k}=1$.
\begin{cor}
\[\tau(k)P_k=-\sigma(k)+\sum_{m \in \mathcal{V}_{k-1}} (-1)^{g_{k-1}(m)} (-3)^{\ell_\mathrm{odd}}+\sum_{m \in \mathcal{V}_k} (-1)^{g_k(m)} (-3)^{\ell_\mathrm{odd}},\]
where ${\ell_\mathrm{odd}}(m)$ counts the odd indices of $m$.
\end{cor}

We conclude with an application of Theorem \ref{thm:EMAnalogue}.
Let $n_k$ and $p_k$ be the number of terms of $R_k$ that have negative sign and positive sign, respectively.  From the recurrence formula \eqref{eqn:TRecurrence}, these sequences satisfy
\[n_k=p_{k-3}+n_{k-3}+2n_{k-2}+n_{k-1},\]
and
\[p_k=n_{k-3}+p_{k-3}+2p_{k-2}+p_{k-1}.\]
Let $J_{k+1}=p_k-n_k$.  Then $J_k$ satisfies the recurrence formula
$J_k=J_{k-1}+2J_{k-2}$,
with initial conditions $J_1=J_2=1$; these are the Jacobsthal numbers.  One property of the Jacobosthal numbers is that $J_k+J_{k-1}=2^{k-1}$ for $k\geq 0$.  Thus from \eqref{eqn:ClassNum},
\begin{equation}\label{eqn:blah}
-\sigma(k)+\sum_{m \in \mathcal{R}_{k-1}} (-1)^{g_{k-1}(m)}+\sum_{m \in \mathcal{R}_k} (-1)^{g_k(m)}=2^k.
\end{equation}

\section{Tables}

The following two figures show the relations between the different polynomials associated
with $K$ encountered in
this paper.

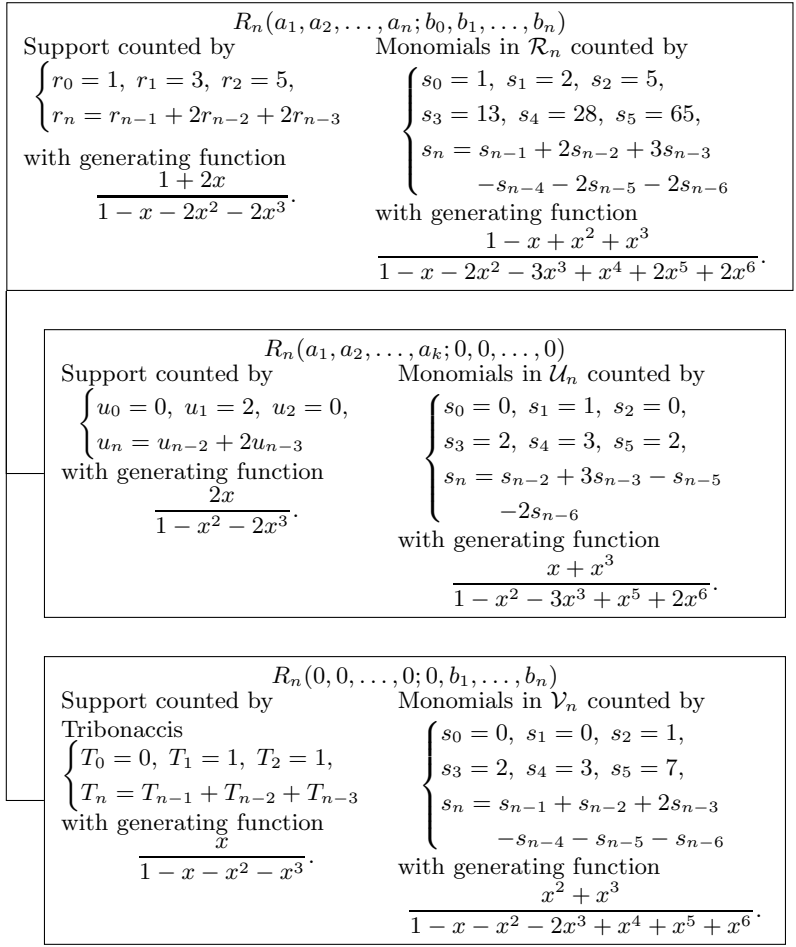
\begin{figure}[h]
\hspace*{\fill}
\begin{tikzpicture}[grow via three points={%
	one child at (1,0) and two children at (.2, -4.35) and (.2, -8.70)},
	edge from parent path={(\tikzparentnode.south west) |- (\tikzchildnode.west)}]
	]
\node [draw, anchor=south west, above] {
\begin{small}
\begin{minipage}{10cm}
\hspace*{\fill} $R_n(a_1,a_2,\dots,a_n;b_0,b_1,\dots,b_n)$ \hspace*{\fill}\\
\raisebox{8pt}{\begin{minipage}[t]{4.6cm}
Support counted by \\
\hspace*{\fill} {$\displaystyle\begin{cases} r_0=1, \; r_1=3, \; r_2=5,\\
r_n=r_{n-1}+2r_{n-2}+2r_{n-3}\end{cases}$} \hspace*{\fill}\\
with generating function \\
\hspace*{\fill} $\dfrac{1+2x}{1-x-2x^2-2x^3}$. \hspace*{\fill}
\end{minipage}}
\raisebox{8pt}{\begin{minipage}[t]{5.4cm}
Monomials in $\mathcal{R}_n$ counted by \\
\hspace*{\fill} {$\displaystyle\begin{cases} s_0=1, \; s_1=2, \; s_2=5, \\
s_3=13, \; s_4=28, \; s_5=65,\\
s_n=s_{n-1}+2s_{n-2}+3s_{n-3}\\
\qquad \; -s_{n-4}-2s_{n-5}-2s_{n-6}\end{cases}$} \hspace*{\fill}\\
with generating function\\
{$\dfrac{1-x+x^2+x^3}{1-x-2x^2-3x^3+x^4+2x^5+2x^6}$.}
\end{minipage}}
\end{minipage}	
\end{small}
}
child { node [draw] {
\begin{small}
\begin{minipage}{9.4cm}
\hspace*{\fill} $R_n(a_1,a_2,\dots,a_k;0,0,\dots,0)$ \hspace*{\fill}\\
\raisebox{8pt}{\begin{minipage}[t]{4.4cm}
Support counted by \\
\hspace*{\fill} {$\displaystyle\begin{cases} 
u_0=0, \; u_1=2, \; u_2=0,\\
u_n=u_{n-2}+2u_{n-3}
\end{cases}$} \hspace*{\fill}\\
with generating function \\
\hspace*{\fill} $\dfrac{2x}{1-x^2-2x^3}$. \hspace*{\fill}
\end{minipage}}
\raisebox{8pt}{\begin{minipage}[t]{5.0cm}
Monomials in $\mathcal{U}_n$ counted by \\
\hspace*{\fill} {$\displaystyle\begin{cases} 
s_0=0, \; s_1=1, \; s_2=0,\\
s_3=2, \; s_4=3, \; s_5=2,\\
s_n=s_{n-2}+3s_{n-3}-s_{n-5}\\
\qquad \; -2s_{n-6}
\end{cases}$} \hspace*{\fill}\\
with generating function\\
\hspace*{\fill} $\dfrac{x+x^3}{1-x^2-3x^3+x^5+2x^6}$. \hspace*{\fill}
\end{minipage}}
\end{minipage}
\end{small}	         
         } }
child { node [draw] {
\begin{small}
\begin{minipage}{9.4cm}
\hspace*{\fill} $R_n(0,0,\dots,0;0,b_1,\dots,b_n)$ \hspace*{\fill}\\
\raisebox{8pt}{\begin{minipage}[t]{4.4cm}
Support counted by \\ Tribonaccis\\
{$\displaystyle\begin{cases} 
T_0=0, \; T_1=1, \; T_2=1,\\
T_n=T_{n-1}+T_{n-2}+T_{n-3}
\end{cases}$}\\
with generating function \\
\hspace*{\fill} $\dfrac{x}{1-x-x^2-x^3}$. \hspace*{\fill}
\end{minipage}}
\raisebox{8pt}{\begin{minipage}[t]{5.0cm}
Monomials in $\mathcal{V}_n$ counted by \\
\hspace*{\fill} {$\displaystyle\begin{cases}
s_0=0, \; s_1=0, \; s_2=1,\\
s_3=2, \; s_4=3, \; s_5=7,\\
s_n=s_{n-1}+s_{n-2}+2s_{n-3}\\
\qquad \; -s_{n-4}-s_{n-5}-s_{n-6}
\end{cases}$} \hspace*{\fill}\\
with generating function\\
\hspace*{\fill} $\dfrac{x^2+x^3}{1-x-x^2-2x^3+x^4+x^5+x^6}$. \hspace*{\fill}
\end{minipage}}
\end{minipage}	   
\end{small}
} }
     ;
\end{tikzpicture}
\hspace*{\fill}
\label{fig:StructurePolys}
\caption{Counting sequences related to polynomials $R_n$ and their 
special cases. For Tribonaccis, see \cite{Feinberg1963}.}
\end{figure}

\pagebreak

\begin{figure}
\hspace*{\fill}
\begin{tikzpicture}[grow via three points={%
	one child at (.8,0) and two children at (.2, -4.38) and (.2, -8.75)},
	edge from parent path={(\tikzparentnode.south west) |- (\tikzchildnode.west)}]
	]
\node [draw, anchor=south west, above] {
\begin{minipage}{10.2cm}
\begin{small}
\hspace*{\fill} $P_n(a_1,a_2,\dots,a_n;b_0,b_1,\dots,b_n)$ \hspace*{\fill}\\
\raisebox{8pt}{\begin{minipage}[t]{4.8cm}
Support counted by \\
\hspace*{\fill} {$\displaystyle\begin{cases} p_0=1, \; p_1=4, \; p_2=8,\\
p_n=p_{n-1}+2p_{n-2}+2p_{n-3}\end{cases}$} \hspace*{\fill}\\
with generating function \\
\hspace*{\fill} {$\dfrac{1+3x+2x^2}{1-x-2x^2-2x^3}$}. \hspace*{\fill}
\end{minipage}}
\raisebox{8pt}{\begin{minipage}[t]{5.4cm}
Monomials in $\mathcal{R}_n\cup\mathcal{R}_{n-1}$ counted by \\
\hspace*{\fill} {$\displaystyle\begin{cases} s_0=1, \; s_1=3, \; s_2=7,\\
s_3=18, \; s_4=41, \; s_5=93,\\
s_n=s_{n-1}+2s_{n-2}+3s_{n-3}\\
\qquad \; -s_{n-4}-2s_{n-5}-2s_{n-6}\end{cases}$} \hspace*{\fill}\\
with generating function\\
{$\dfrac{1+2x+2x^2+2x^3+x^4}{1-x-2x^2-3x^3+x^4+2x^5+2x^6}$.}
\end{minipage}}
\end{small}
\end{minipage}	
}
child { node [draw] {
\begin{minipage}{9.6cm}
\begin{small}
\hspace*{\fill} $P_n(a_1,a_2,\dots,a_k;0,0,\dots,0)$ \hspace*{\fill}\\
\raisebox{8pt}{\begin{minipage}[t]{4.4cm}
Support counted by \\
\hspace*{\fill} {$\displaystyle\begin{cases} 
p_0=0, \; p_1=2, \; p_2=2,\\
p_n=p_{n-2}+2p_{n-3}
\end{cases}$} \hspace*{\fill}\\
with generating function \\
\hspace*{\fill} $\dfrac{2x+2x^2}{1-x^2-2x^3}$. \hspace*{\fill}
\end{minipage}}
\raisebox{8pt}{\begin{minipage}[t]{5.2cm}
Monomials in $\mathcal{U}_n\cup\mathcal{U}_{n-1}$ counted by \\
\hspace*{\fill} {$\displaystyle\begin{cases} 
s_0=0, \; s_1=1, \; s_2=1,\\
s_3=2, \; s_4=5, \; s_5=5,\\
s_n=s_{n-2}+3s_{n-3}-s_{n-5}\\
\qquad \; -2s_{n-6}
\end{cases}$} \hspace*{\fill}\\
with generating function\\
\hspace*{\fill} $\dfrac{x+x^2+x^3}{1-x^2-3x^3+x^5+2x^6}$. \hspace*{\fill}
\end{minipage}}
\end{small}
\end{minipage}	         
         } }
child { node [draw] {
\begin{minipage}{9.6cm}
\begin{small}
\hspace*{\fill} $P_n(0,0,\dots,0;0,b_1,\dots,b_n)$ \hspace*{\fill}\\
\raisebox{8pt}{\begin{minipage}[t]{4.4cm}
Support counted by \\ Tribonacci-type sequence\\
\hspace*{\fill} {$\displaystyle\begin{cases} 
p_0=0, \; p_1=1, \; p_2=2,\\
p_n=p_{n-1}+p_{n-2}+p_{n-3}
\end{cases}$} \hspace*{\fill}\\
with generating function \\
\hspace*{\fill} $\dfrac{x+x^2}{1-x-x^2-x^3}$. \hspace*{\fill}
\end{minipage}}
\raisebox{8pt}{\begin{minipage}[t]{5.2cm}
Monomials in $\mathcal{V}_n\cup\mathcal{V}_{n-1}$ counted by \\
\hspace*{\fill} {$\displaystyle\begin{cases}
s_0=0, \; s_1=0, \; s_2=1,\\
s_3=3, \; s_4=5, \; s_5=10,\\
s_n=s_{n-1}+s_{n-2}+2s_{n-3}\\
\qquad \;-s_{n-4}-s_{n-5}-s_{n-6}
\end{cases}$} \hspace*{\fill}\\
with generating function\\
\hspace*{\fill} $\dfrac{x^2+2x^3+x^4}{1-x-x^2-2x^3+x^4+x^5+x^6}$. \hspace*{\fill}
\end{minipage}} 
\end{small}
\end{minipage}	   
} }
     ;
\end{tikzpicture}
\hspace*{\fill}
\label{fig:ClassNums}
\caption{Counting sequences related to classical numerators $P_n$ and their special cases. Note that the generating functions are a product of $(1+x)$ and the generating functions of the polynomials in the previous figure. }
\end{figure}

\clearpage

\bibliography{combinatorics_of_continuants}
\bibliographystyle{plain}

\end{document}